\documentclass{amsart}

\usepackage[normalem]{ulem}
\usepackage{amssymb,epsfig,mathrsfs,mathpazo,stackengine,scalerel,url,amsthm,thmtools,bbm}
\usepackage{xcolor}
\usepackage{stmaryrd}
\usepackage{epsfig,bm}
\usepackage[multiple]{footmisc}
\usepackage{accents} 
\usepackage{mathtools} 
\newcommand{\osc}{{\rm osc}}

\usepackage{arydshln, booktabs, makecell, pbox, tabularx}

\usepackage{graphicx, caption, subcaption}

\newtheorem{theorem}{Theorem}[section]
\newtheorem{lemma}[theorem]{Lemma}
\newtheorem{proposition}[theorem]{Proposition}

\declaretheorem[style=definition,qed=$\vartriangle$,sibling=theorem]{example}
\declaretheorem[style=remark,qed=$\vartriangle$,sibling=theorem]{remark}
\numberwithin{equation}{section}

\newcommand{\R}{\mathbb R}

\newcommand{\N}{\mathbb N}

\newcommand{\cF}{\mathcal F}
\newcommand{\cP}{\mathcal P}

\newcommand{\cE}{\mathcal E}

\newcommand{\cL}{\mathcal L}
\newcommand{\cS}{\mathcal S}
\newcommand{\Lis}{\cL\mathrm{is}}

\newcommand{\identity}{\mathrm{Id}}

\DeclareMathOperator{\ran}{ran}

\DeclareMathOperator*{\argmin}{argmin}

\DeclareMathOperator{\divv}{div}

\newcommand{\nrm}{| \! | \! |}

\newcommand{\new}[1]{{\color{black}{#1}}}

\newcommand{\be}{\begin{equation}}
\newcommand{\ee}{\end{equation}}

\newcommand{\tria}{{\mathcal T}}
\newcommand{\RT}{\mathit{RT}}
\newcommand{\CR}{\mathit{CR}}

\newcommand{\uumlaut}{{\"u}}

{\par\begin{samepage}%
\begin{tabbing}\ttfamily%
 \hspace*{5mm}\=\hspace{3ex}\=\hspace{3ex}\=\hspace{3ex}\=\hspace{3ex}%
\=\hspace{3ex}\=\hspace{3ex}\=\hspace{3ex}\=\hspace{3ex}\kill}%
{\end{tabbing}\end{samepage}}

\newcounter{ccondition}

\newenvironment{ccondition}{%
\refstepcounter{ccondition}%
\newcommand{\temp}{\theequation}%
\renewcommand{\theequation}{C.\arabic{ccondition}}%
\noindent\ignorespaces%
}%
{%
\renewcommand{\theequation}{\temp}%
\addtocounter{equation}{-1}%
\par\noindent%
\ignorespacesafterend%
}

\newcounter{mylistcounter}
\renewcommand{\themylistcounter}{(\roman{mylistcounter})}

\newenvironment{mylist}{
\begin{list}{\themylistcounter.}{\usecounter{mylistcounter}
\setlength{\labelwidth}{-\smallskipamount}
\setlength{\labelsep}{\medskipamount}
\setlength{\topsep}{\smallskipamount}
\setlength{\itemsep}{\smallskipamount}
\setlength{\itemindent}{0cm}
\setlength{\leftmargin}{0cm}}}
{\end{list}}

\makeatletter
\@namedef{subjclassname@2020}{%
  \textup{2020} Mathematics Subject Classification}
\makeatother

\title{Minimal residual methods in negative or fractional Sobolev norms}

\date{\today}

\author{Harald Monsuur}
\author{Rob Stevenson}
\address{Korteweg-de Vries (KdV) Institute for Mathematics, University of Amsterdam, P.O. Box 94248, 1090 GE Amsterdam, The Netherlands.}
\email{h.monsuur@uva.nl, rob.p.stevenson@gmail.com}

\author{Johannes Storn}
\address{Department of Mathematics, University of Bielefeld, Postfach 10 01 31, 33501 Bielefeld, Germany}
\email{jstorn@math.uni-bielefeld.de}

\thanks{This research has been supported by the Netherlands Organization for Scientific Research (NWO) under contract.~no.~SH-208-11, by the
NSF Grant DMS ID 1720297, and by the Deutsche Forschungsgemeinschaft (DFG,
German Research Foundation) – SFB 1283/2 2021 – 317210226.}

\subjclass[2020]{
35B35, 
35B45, 
65N30, 
}

\keywords{Least squares methods, Fortin interpolator, a posteriori error estimator, inhomogeneous boundary conditions, quasi-optimal approximation}

\begin{document}

\begin{abstract} For numerical approximation the reformulation of a PDE as a residual  minimisation problem has the advantages that the resulting linear system is symmetric positive definite, and that the norm of the residual provides an a posteriori error estimator. Furthermore, it allows for the treatment of general inhomogeneous boundary conditions. In many minimal residual formulations, however, one or more terms of the residual are measured in negative or fractional Sobolev norms. In this work, we provide a general approach to replace those norms by efficiently evaluable expressions without sacrificing quasi-optimality of the resulting numerical solution. We exemplify our approach by verifying the necessary inf-sup conditions for four formulations of a model second order elliptic equation with inhomogeneous Dirichlet and/or Neumann boundary conditions. 
We report on numerical experiments for the Poisson problem with mixed inhomogeneous Dirichlet and Neumann boundary conditions in an ultra-weak first order system formulation.
\end{abstract}

\maketitle
\section{Introduction}
This paper is about minimal residual, or least-squares discretisations of boundary value problems.
We will use the acronym MINRES, despite its common use to denote a certain Krylov subspace iteration.
In an abstract setting, for some Hilbert spaces $X$ and $V$, for convenience over $\R$, an operator $G \in \Lis(X,V)$, and an $f \in V$, we consider the equation
$$
G u = f.
$$
With the notation $G \in \Lis(X,V)$, we mean that $G$ is a boundedly invertible linear operator $X \rightarrow V$, i.e., $G \in \cL(X,V)$ and $G^{-1} \in \cL(V,X)$.

For any closed, in applications finite dimensional subspace $X^\delta \subset X$, let
$$
u^\delta:=\argmin_{w \in X^\delta} \tfrac12 \|G w -f\|^2_V.
$$
This $u^\delta$ is the unique solution of the corresponding Euler-Lagrange equations
\be \label{eq:EL}
\langle G u^\delta, G v\rangle_V=\langle f, G v\rangle_V\quad (v \in X^\delta).
\ee
The bilinear form at the left hand side is bounded, symmetric, and coercive, so that
\be \label{eq:quasi-best}
\|u-u^\delta\|_X \leq \|G\|_{\cL(X,V)} \|G^{-1}\|_{\cL(V,X)}\inf_{w \in X^\delta}\|u-w\|_X,
\ee
i.e., $u^\delta$ is a \emph{quasi-optimal} approximation to $u$ from $X^\delta$.

Additional advantages of a MINRES discretisation are that the system matrix resulting from \eqref{eq:EL} is always symmetric positive definite, and that the method comes with an efficient and reliable computable a posteriori error estimator
$$
\|f-G u^\delta\|_V \in \big[\|G^{-1}\|^{-1}_{\cL(V,X)} \|u-u^\delta\|_X, \|G\|_{\cL(X,V)} \|u-u^\delta\|_X\big].
$$

For more information about MINRES discretisations we refer to the monograph \cite{23.5}, where apart from general theory, many applications are discussed, including (but not restricted to) scalar second order elliptic boundary value problems, Stokes equations, and the equations of linear elasticity.

As explained in \cite[\S2.2.2]{23.5}, for a MINRES discretisation to be competitive it should be `practical'. With that it is meant that $V$ should not be a  fractional 
or negative order Sobolev space, or when it is a Cartesian product, neither of its components should be of that kind, and at the same time $X$ should not be a
Sobolev space of order two (or higher) because that would require a globally $C^1$ finite element subspace $X^\delta$.
In view of these requirements, a first natural step is to write a 2nd order PDE under consideration as a first order system.
It turns out, however, that even then in many applications one or more components of $V$ are fractional 
or negative order Sobolev spaces. 

First the imposition of inhomogeneous boundary conditions lead to residual terms that are measured in fractional Sobolev spaces.
Although the capacity to handle inhomogeneous boundary conditions is often mentioned as an advantage of MINRES methods, until now a fully satisfactory solution how to deal with fractional Sobolev spaces seems not to be available.
Second, if one prefers to avoid an additional regularity condition on the forcing term required for the standard `practical' first order system formulation, one ends up with a residual that is measured in a negative Sobolev norm.  
Finally, more than one dual norms occur with ultra-weak first order formulations which for example are useful to construct `robust' discretisations for Helmholtz equations (\cite{64.155,204.18}).

In \cite{23.5} several possibilities are discussed to find a compromise between having norm equivalence, and so quasi-optimality, and `practicality', for example by replacing negative or fractional Sobolev norms in the MINRES formulation by mesh-dependent weighted $L_2$-norms.
The topic of the current paper is the replacement of negative or fractional Sobolev norms by computable quantities whilst fully retaining quasi-optimality of the MINRES method.
\medskip

This paper is organized as follows. In Sect.~\ref{sec:examples} we give several examples of MINRES formulations of a model scalar second order elliptic boundary value problem, where except for one formulation, one or more terms of the residual are measured in fractional or negative Sobolev spaces.
In an abstract setting in Sect.~\ref{sec:approach} it is shown how such `impractical' MINRES formulations can be turned into `practical' ones without compromising quasi-optimality. 
For the examples from Sect.~\ref{sec:examples}, in Sect.~\ref{sec:infsup}  we verify (uniform) inf-sup conditions that are needed for the conversion of the `impractical' to a `practical' MINRES formulation.
In this section, we also discuss alternative approaches to handle dual norms (\cite{35.83}), or to handle singular forcing terms in an already `practical' MINRES discretisation (\cite{75.259, 75.068}).
In Sect.~\ref{sec:numerics} we illustrate the theoretical findings with some numerical results, and a conclusion is presented in Sect.~\ref{sec:conclusion}.
\medskip

In this paper, by the notation $C \lesssim D$ we will mean that $C$ can be bounded by a multiple of $D$, independently of parameters which $C$ and $D$ may depend on,
as the discretisation index $\delta$.
Obviously, $C \gtrsim D$ is defined as $D \lesssim C$, and $C\eqsim D$ as $C\lesssim D$ and $C \gtrsim D$.

\section{Examples of MINRES discretisations} \label{sec:examples}
The results from this section that concern well-posedness of MINRES formulations, i.e., boundedly invertibility of the operator $G$, for the case of essential inhomogeneous boundary conditions are taken from \cite{249.96}.
The key to arrive at those results was a lemma that, in the slightly modified version from \cite[Lemma 2.7]{75.28}, is recalled below.
\begin{lemma} \label{lem:bi}
Let $X$ and $V_2$ be Banach spaces, and $V_1$ be a normed linear space.
Let $T \in \cL(X,V_2)$ be surjective, and let $G \in \cL(X,V_1)$ be such that 
 $G|_{\ker T} \in \Lis(\ker T,V_1)$.
Then $(G, T) \in  \Lis\big(X,V_1\times V_2\big)$.
\end{lemma}

On a bounded Lipschitz domain $\Omega \subset \R^d$, where $d \geq 2$, and closed $\Gamma_D, \Gamma_N \subset \partial\Omega$, with $\Gamma_D \cup \Gamma_N =\partial\Omega$ and $|\Gamma_D \cap \Gamma_N|=0$, we consider the following boundary value problem
\begin{equation} \label{bvp}
 \left\{
\begin{array}{r@{}c@{}ll}
-{\rm div}\, A \nabla u+ B u &\,\,=\,\,& g &\text{ on } \Omega,\\
u &\,\,=\,\,& h_D &\text{ on } \Gamma_D,\\
\vec{n}\cdot A \nabla u &\,\,=\,\,& h_N &\text{ on } \Gamma_N,
\end{array}
\right.
\end{equation}
 where $\vec{n}$ is the outward pointing unit vector normal to the boundary, $B$ is a bounded linear partial differential operator of at most first order, i.e.,
\begin{ccondition}
\begin{equation} \label{4}
B \in \cL(H^1(\Omega),L_2(\Omega)),
\end{equation}
\end{ccondition}
 and $A(\cdot) \in L_\infty(\Omega)^{d\times d}$ is real, symmetric with
\begin{ccondition}
$$
\xi^\top A(\cdot) \xi \eqsim \|\xi\|^2 \quad (\xi \in \R^d).
$$
\end{ccondition}

We assume that the standard variational formulation of \eqref{bvp} for the case of homogeneous Dirichlet boundary conditions is well-posed, i.e., with
$H^1_{0,\Gamma_D}(\Omega):=\{v \in H^1(\Omega) \colon \gamma_D v=0\}$, where $\gamma_D$ is the trace operator on $\Gamma_D$,
the operator
\begin{ccondition}
\begin{equation} \label{11}
G:=w\mapsto (v \mapsto \int_\Omega A \nabla w \cdot \nabla v +B w \, v \,dx)
\in \Lis\big(H^1_{0,\Gamma_D}(\Omega),H^1_{0,\Gamma_D}(\Omega)'\big).\footnotemark
\end{equation}
\end{ccondition}
\footnotetext{In the case that $\Gamma_D=\emptyset$, it can be needed, as when $B=0$, to replace $H^1_{0,\Gamma_D}(\Omega)=H^1(\Omega)$ by $H^1(\Omega)/\R$. For simplicity, we do not consider this situation.}%
With this standard variational formulation, the Neumann boundary condition is natural, and the Dirichlet boundary condition is essential.
We are ready to give the first example of a MINRES discretisation.

\begin{example}[2nd order weak formulation] \label{ex:2ndorder} 
Let $g\in H^1_{0,\Gamma_D}(\Omega)'$ and $h_N\in H^{-\frac12}(\Gamma_N)=H_{00}^{\frac12}(\Gamma_N)'$, where $H_{00}^{\frac12}(\Gamma_N)=[L_2(\Gamma_N),H^1_0(\Gamma_N)]_{\frac12,2}$, so that consequently
$f:=v \mapsto  g(v)+\int_{\Gamma_N} h_N v \,ds \in H^1_{0,\Gamma_D}(\Omega)'$.
\begin{mylist}
\item \label{ex:2ndorder1} Let $h_D = 0$ (or $\Gamma_D=\emptyset$). For any finite dimensional subspace $X^\delta \subset H^1_{0,\Gamma_D}(\Omega)$, \eqref{11} shows that a 
quasi-optimal MINRES approximation to the solution of \eqref{bvp} is
$$
u^\delta:=\argmin_{w \in X^\delta} \tfrac12 \|G w -f\|^2_{H^1_{0,\Gamma_D}(\Omega)'}.
$$
\item \label{ex:2ndorder2} Let $0 \neq h_D \in H^{\frac12}(\Gamma_D)$. By surjectivity of $\gamma_D \in \cL\big(H^1(\Omega),H^{\frac12}(\Gamma_D)\big)$, Lemma~\ref{lem:bi} shows that
$(G,\gamma_D) \in \Lis\big(H^1(\Omega),H^1_{0,\Gamma_D}(\Omega)' \times H^{\frac12}(\Gamma_D)\big)$, so that for any finite dimensional subspace  $X^\delta \subset H^1(\Omega)$,
\be \label{eq:2ndorder}
u^\delta:=\argmin_{w \in X^\delta} \tfrac12 \big(\|G w -f\|^2_{H^1_{0,\Gamma_D}(\Omega)'}+\|\gamma_D w-h_D\|_{H^{\frac12}(\Gamma_D)}^2\big)
\ee
is a quasi-optimal MINRES approximation to the solution of \eqref{bvp}.\qedhere
\end{mylist}
\end{example}

Introducing $\vec{p}=A \nabla u$, for the remaining examples we consider the reformulation of \eqref{bvp} as the first order system 
\begin{equation} \label{fos}
 \left\{
\begin{array}{r@{}c@{}ll}
\vec{p}-A \nabla u &\,\,=\,\,& 0 &\text{ on } \Omega,\\
B u -{\rm div}\, \vec{p}&\,\,=\,\,& g &\text{ on } \Omega,\\
u &\,\,=\,\,& h_D &\text{ on } \Gamma_D,\\
\vec{p} \cdot \vec{n}&\,\,=\,\,& h_N &\text{ on } \Gamma_N.
\end{array}
\right.
\end{equation}

By measuring the residuals of the first two equations in \eqref{fos} in the `mild' $L_2(\Omega)$-sense, we obtain the following first order system MINRES or FOSLS discretisation.
Both Dirichlet and Neumann boundary conditions are essential ones.

\begin{example}[mild formulation] \label{ex:mild} Let $g \in L_2(\Omega)$.
\begin{mylist}
\item \label{mild1} Let $h_D = 0$ (or $\Gamma_D=\emptyset$), and $h_N =0$ (or $\Gamma_N=\emptyset$). As shown in \cite[Thm.~3.1]{249.96}, the  operator
\begin{align*}
G:= (\vec{q},w)&\mapsto (\vec{q}-A \nabla w,B w -\divv \vec{q})\\
& \in \Lis\big(H_{0,\Gamma_N}({\rm div};\Omega)\times H^1_{0,\Gamma_D}(\Omega), L_2(\Omega)^d \times L_2(\Omega)\big),
\end{align*}
and so for any finite dimensional subspace $X^\delta \subset H_{0,\Gamma_N}({\rm div};\Omega)\times H^1_{0,\Gamma_D}(\Omega)$,
$$
(\vec{p}^\delta,u^\delta):=\argmin_{(\vec{q},w) \in X^\delta} \tfrac12 \|G(\vec{q},w)-(0,g) \|^2_{L_2(\Omega)^d \times L_2(\Omega)}
$$
is a quasi-optimal MINRES approximation to the solution of \eqref{fos}.
\item  \label{mild2}  Let $0 \neq h_D \in H^{\frac12}(\Gamma_D)$ and $0 \neq h_N  \in H^{-\frac{1}{2}}(\Gamma_N)$.\footnote{The cases that either $h_D \neq 0$ \emph{or} $h_N  \neq 0$ cause no additional difficulties}
From the surjectivity of the pair of normal trace and trace operators on $\Gamma_N$ or $\Gamma_D$
$$
(\gamma_N,\gamma_D) \in \cL(H(\divv;\Omega)\times H^1(\Omega), H^{-\frac{1}{2}}(\Gamma_N) \times H^{\frac{1}{2}}(\Gamma_D)),
$$
Lemma~\ref{lem:bi} shows that for any finite dimensional subspace  $X^\delta \subset H(\divv;\Omega)\times H^1(\Omega)$,
\begin{align*}
(\vec{p}^\delta,u^\delta):=\argmin_{(\vec{q},w) \in X^\delta} \tfrac12\big( &\|G(\vec{q},w)-(0,g) \|^2_{L_2(\Omega)^d \times L_2(\Omega)}\\
&+
\|\gamma_N \vec{q}-h_N\|_{H^{-\frac{1}{2}}(\Gamma_N) }^2+\|\gamma_D w-h_D\|_{H^{\frac{1}{2}}(\Gamma_D) }^2\big)
\end{align*}
is a quasi-optimal MINRES approximation to the solution of \eqref{fos}.\qedhere
\end{mylist}
\end{example}

Among the known MINRES formulations of \eqref{bvp}, the formulation from Example~\ref{ex:mild}\ref{mild1} (so for homogeneous boundary conditions) is the only one that is `practical' because the residual is minimized in $L_2$-norm.
A disadvantage of this mild formulation is that it only applies to a forcing term $g \in L_2(\Omega)$, whilst the $H(\divv;\Omega)$-norm instead of the more natural $L_2(\Omega)^d$-norm in which the error in $\vec{p}=A \nabla u$ is measured requires additional smoothness of $u$ to guarantee a certain convergence rate.

These disadvantages vanish in the following mild-weak formulation, which, however, in unmodified form is impractical.
Another approach to overcome the disadvantages of the mild formulation, which is presented in \cite{75.259, 75.068}, is to replace in the least squares minimization the forcing term $g$ by a finite element approximation.  Later in Remark~\ref{rem:singulardata}, we discuss this idea in detail.

In the following mild-weak formulation the second equation in \eqref{fos} is imposed in an only weak sense.
It has the consequence that the Neumann boundary condition is a natural one.

\begin{example}[mild-weak formulation] \label{ex:mild-weak} 
Let $g\in H^1_{0,\Gamma_D}(\Omega)'$ and $h_N\in H^{-\frac12}(\Gamma_N)$, so that 
$f:=v \mapsto  g(v)+\int_{\Gamma_N} h_N v \,ds \in H^1_{0,\Gamma_D}(\Omega)'$.
\begin{mylist}
\item  \label{ex:mild-weak1}
Let $h_D = 0$ (or $\Gamma_D=\emptyset$). As shown in \cite{35.835}, the operator
$$
G=(G_1,G_2):= (\vec{q},w)\mapsto \big(\vec{q}-A\nabla w,v \mapsto \int_\Omega \vec{q} \cdot \nabla v+ B w\,v \,dx\big)
$$
satisfies
\be \label{eq:equiv}
\|G(\vec{q},w)\|_{L_2(\Omega)^d \times H^1_{0,\Gamma_D}(\Omega)'} \eqsim \|(\vec{q},w)\|_{L_2(\Omega)^d \times H^1(\Omega)}
\ee
($(\vec{q},w) \in L_2(\Omega)^d \times H^1_{0,\Gamma_D}(\Omega)$).
It remains to verify surjectivity. 
Given $(\vec{r},f) \in L_2(\Omega)^d \times H^1_{0,\Gamma_D}(\Omega)'$, \eqref{11} shows that there exists a $w \in H^1_{0,\Gamma_D}(\Omega)$ with 
$$
\int_\Omega A\nabla w \cdot\nabla v+B w \,v\,dx=f(v)-\int_\Omega \vec{r}\cdot\nabla v\,dx \quad(v \in H^1_{0,\Gamma_D}(\Omega)).
$$ 
With $\vec{q}:=\vec{r}+A \nabla w$, we conclude that $G(\vec{q},w)=(\vec{r},f)$.
Surjectivity with \eqref{eq:equiv} implies that $G \in \Lis\big(L_2(\Omega)^d \times H^1_{0,\Gamma_D}(\Omega),L_2(\Omega)^d \times H^1_{0,\Gamma_D}(\Omega)'\big)$.
So for any finite dimensional subspace $X^\delta \subset L_2(\Omega)^d \times H^1_{0,\Gamma_D}(\Omega)$,
$$
(\vec{p}^\delta,u^\delta):=\argmin_{(\vec{q},w) \in X^\delta} \tfrac12 \big(
\|G_1(\vec{q},w) \|^2_{L_2(\Omega)^d}
+
\|G_2(\vec{q},w)-f\|^2_{H^1_{0,\Gamma_D}(\Omega)'}\big)
$$
is a quasi-optimal MINRES approximation to the solution of \eqref{fos}.
\item \label{ex:mild-weak2}  Let $0 \neq h_D \in H^{\frac12}(\Gamma_D)$. 
From $L_2(\Omega) \times H^1(\Omega) \rightarrow H^{\frac12}(\Gamma_D)\colon (\vec{q},w)\mapsto \gamma_D w$ being surjective,
Lemma~\ref{lem:bi} shows that for any finite dimensional subspace  $X^\delta \subset L_2(\Omega)^d \times H^1(\Omega)$,
\begin{align*}
(\vec{p}^\delta,u^\delta):=\argmin_{(\vec{q},w) \in X^\delta} \tfrac12\big(
\|G_1(\vec{q},w) \|^2_{L_2(\Omega)^d}
&+
\|G_2(\vec{q},w)-f\|^2_{H^1_{0,\Gamma_D}(\Omega)'}\\
&+
\|\gamma_D w-h_D\|_{H^{\frac{1}{2}}(\Gamma_D) }^2\big)
\end{align*}
is a quasi-optimal MINRES approximation to the solution of \eqref{fos}.\qedhere
\end{mylist}
\end{example}

Finally, by imposing both the first and second equation in \eqref{fos} in a weak sense we obtain the ultra-weak formulation.
In order to do so, first we  specify the operator $B$ from \eqref{4} to $B:= w \mapsto \vec{b}\cdot \nabla w+c w$ for some $\vec{b} \in L_\infty(\Omega)^d$ and $c \in L_\infty(\Omega)$, and, to avoid additional smoothness conditions on $\vec{b}$, write the second equation in \eqref{fos} as $\vec{b}\cdot A^{-1} \vec{p}+c u -\divv \vec{p}=g$.

\begin{example}[ultra-weak formulation] \label{ex:ultra-weak} 
Let $h_D\in H^{\frac12}(\Gamma_D)$, so that $f_1:=\vec{z}\mapsto \int_{\Gamma_D} h_D \vec{z}\cdot\vec{n}\,ds \in H_{0,\Gamma_N}(\divv;\Omega)'$, 
and let $g\in H^1_{0,\Gamma_D}(\Omega)'$, and $h_N\in H^{-\frac12}(\Gamma_N)$, so that $f_2:=v \mapsto  g(v)+\int_{\Gamma_N} h_N v \,ds \in H^1_{0,\Gamma_D}(\Omega)'$.
As shown in \cite[Thm.~3.3]{249.96},
\begin{align*}
G:= (\vec{q},w)\mapsto 
\Big(
\vec{z} \mapsto \int_{\Omega}
 A^{-1} \vec{q}\cdot \vec{z}+w \divv \vec{z}\,dx,
 v \mapsto \int_{\Omega} (\vec{b} \cdot A^{-1} \vec{q}+c w) v+\vec{q} \cdot \nabla v\,dx
 \Big)\\
 \in \Lis\big(L_2(\Omega)^d \times L_2(\Omega), H_{0,\Gamma_N}(\divv;\Omega)' \times H^1_{0,\Gamma_D}(\Omega)'\big).
\end{align*}
Consequently, for any finite dimensional subspace $X^\delta \subset L_2(\Omega)^d \times L_2(\Omega)$,
$$
(\vec{p}^\delta,u^\delta):=\argmin_{(\vec{q},w) \in X^\delta} \tfrac12 \|G(\vec{q},w)-(f_1,f_2) \|^2_{H_{0,\Gamma_N}(\divv;\Omega)'  \times H^1_{0,\Gamma_D}(\Omega)'}
$$
is a quasi-optimal MINRES approximation to the solution of \eqref{fos}.
\end{example}

\section{Turning an impractical MINRES formulation into a practical one} \label{sec:approach}
\subsection{Dealing with a dual norm}
In our examples, the MINRES discretisations are of the form
\be \label{eq:minres}
u^\delta:=\argmin_{z \in X^\delta} \tfrac12 \big(\sum_{i=1}^k \|G_i z-f_i\|_{Y_i'}^2+\sum_{i=k+1}^m \|G_i z-f_i\|_{Y_i}^2\big)
\ee
with $0 \leq k \leq m$, $m\geq 1$, Hilbert spaces $X$ and $(Y_i)_{1 \leq i \leq m}$,
 $G=(G_i)_{1 \leq i \leq m} \in \Lis(X,Y'_1\times\cdots\times Y'_k\times Y_{k+1} \times \cdots \times Y_m)$, and a finite dimensional subspace $X^\delta \subset X$, and where, for $1 \leq i \leq k$, the spaces $Y_i$ are such that the Riesz map $Y_i' \rightarrow Y_i$ cannot be efficiently evaluated (i.e., $Y_i$ is not an $L_2$-space).

In Examples~\ref{ex:2ndorder}\ref{ex:2ndorder2}, \ref{ex:mild}\ref{mild2}, and \ref{ex:mild-weak}\ref{ex:mild-weak2}, we furthermore encountered a residual component that was measured
in $\|\cdot\|_{H^{\frac12}(\Gamma_D)}$, which norm cannot be efficiently evaluated. By writing $\|\cdot\|_{H^{\frac12}(\Gamma_D)}=\|\cdot\|_{\widetilde{H}^{-\frac12}(\Gamma_D)'}$, where $\widetilde{H}^{-\frac12}(\Gamma_D):=H^{\frac12}(\Gamma_D)'$, and handling analogously for all Sobolev norms with positive fractional orders, we may assume that all \emph{non-dual} norms $\|\cdot\|_{Y_i}$ in \eqref{eq:minres} are efficiently evaluable.

\begin{remark} A previously proposed approach to deal with $\|\cdot\|_{H^{\frac12}(\Gamma_D)}$ is to replace it by an efficiently evaluable semi-norm that on a selected finite element subspace is equivalent to $\|\cdot\|_{H^{\frac12}(\Gamma_D)}$ (see \cite{249.06}). The so modified least squares functional is then only equivalent to the original one modulo a data-oscillation term, so that quasi-optimality is not guaranteed.
\end{remark}

The dual norms $\|\cdot\|_{Y_i'}$ for $1 \leq i \leq k$ in \eqref{eq:minres} cannot be evaluated, which makes the discretisation \eqref{eq:minres} impractical.
To solve this, we will select finite dimensional subspaces $Y_i^\delta = Y_i^\delta(X^\delta) \subset Y$ such that
\be \label{eq:defgamma}
\gamma_i^\delta := \inf_{\{z \in X^\delta\colon G_i z \neq 0\}}\frac{\sup_{0 \neq y_i \in Y^\delta_i} \frac{|(G_i z)(y_i)|}{\|y_i\|_{Y_i}}}{\|G_iz\|_{Y_i'}} >0,
\ee
and \emph{replace} the MINRES discretisation \eqref{eq:minres} by 
\be \label{eq:minresprac}
u^\delta:=\argmin_{z \in X^\delta} \tfrac12 \big(\sum_{i=1}^k \sup_{0 \neq y_i \in Y_i^\delta} \frac{|(G_i z-f_i)(y_i)|^2}{\|y_i\|_{Y_i}^2}+
\sum_{i=k+1}^m \|G_i z-f_i\|_{Y_i}^2\big).
\ee

To analyze \eqref{eq:minresprac}, for notational convenience in the remainder of this subsection for $k+1 \leq i\leq m$
 we rewrite
$\|G_i z-f_i\|_{Y_i}^2$ as $\|R_i^{-1}(G_i z-f_i)\|_{Y'_i}^2$, where $R_i\in \Lis(Y_i',Y_i)$ is the Riesz map defined by $f(v)=\langle R_i f,v\rangle_{Y_i}$.
Redefining, for $k+1 \leq i \leq m$,
$G_i:=R_i^{-1} G_i$ and $f_i:=R_i^{-1} f_i$, and setting $Y_i^\delta:=Y_i$ (so that $\gamma_i=1$), with
$G:=(G,\ldots,G_m)$, $f:=(f_1,\ldots,f_m)$, $Y^\delta:=Y^\delta_1\times \cdots \times Y^\delta_m$, $Y:=Y_1\times \cdots \times Y_m$,
the solution of \eqref{eq:minresprac} is equivalently given by
\be \label{eq:minresprac2}
u^\delta:=\argmin_{z \in X^\delta} \tfrac12 \sup_{0 \neq y \in Y^\delta} \frac{|(G z-f)(y)|^2}{\|y\|_{Y}^2}.
\ee
With the newly defined $(G_i)_{k+1 \leq i \leq m}$, we have 
$G \in \Lis(X,Y')$.

\begin{lemma} \label{lem1} With $G$ and $(\gamma_i^\delta)_{1 \leq i \leq m}$ defined above, 
and 
$$
\gamma^\delta := \inf_{\{z \in X^\delta\colon G z \neq 0\}}\frac{\sup_{0 \neq y \in Y^\delta} \frac{|(G z)(y)|}{\|y\|_{Y}}}{\|G z\|_{Y'}},
$$
it holds that $\gamma^\delta \geq \min_{1 \leq i \leq m} \gamma^\delta_i$.
\end{lemma}

\begin{proof} For each $z \in X^\delta$, for $1 \leq i \leq m$ there exists a $y_i \in Y_i^\delta$ with $\|y_i\|_{Y_i}=\|G_i z\|_{Y_i'}$ and $(G_i z)(y_i) \geq  \gamma_i^\delta\|G_i z\|_{Y_i'}^2$. So for $y:=(y_i)_{1 \leq i \leq m} \in Y^\delta$,
$$
(Gz)(y) =\sum_{i=1}^m G_i(z)(y_i) \geq \min_{1 \leq i \leq m} \gamma_i^\delta \sum_{i=1}^m \|G_i z\|_{Y_i'}^2= \min_{1 \leq i \leq m} \gamma_i^\delta \,\|Gz\|_{Y'}\|y\|_Y,
$$
which completes the proof.
\end{proof}

\begin{theorem} \label{thm:quasi-opt} Let $\gamma^\delta>0$. Setting $\nrm \cdot\nrm_X:=\|G \cdot\|_{Y'}$, for $u = G^{-1} f$ and $u^\delta$ from \eqref{eq:minresprac2}, it holds that
\be \label{eq:gamma}
 \inf_{u \in X \setminus X^\delta} \frac{\inf_{w \in X^\delta}\nrm u-w\nrm_X}{\nrm u-u^\delta\nrm_X} \,\,=\,\, \gamma^\delta,
\ee 
and so
\be \label{eq:upp}
\|u-u^\delta\|_X \leq \tfrac{\|G\|_{\cL(X,Y')}\|G^{-1}\|_{\cL(Y',X)}}{\gamma^\delta}\inf_{w \in X^\delta}\| u-w\|_X.
\ee
\end{theorem}

\begin{proof} First we recall from \cite[Prop.~2.2]{35.8565} (building on the seminal work \cite{64.14}), that the MINRES discretisation \eqref{eq:minresprac2} can equivalently be written as a Petrov-Galerkin discretisation:
With $R^\delta \in \Lis({Y^\delta}',Y^\delta)$ defined by $f(v)=\langle R^\delta f,v\rangle_Y$, we have
$$
\tfrac12 \sup_{0 \neq y \in Y^\delta} \frac{|(G z-f)(y)|^2}{\|y\|_{Y}^2}=
\tfrac12 \sup_{0 \neq y \in Y^\delta} \frac{\langle R^\delta (G z-f),y\rangle_Y^2}{\|y\|_{Y}^2}=\tfrac12 \|R^\delta(G z-f)\|_Y^2,
$$
so that \eqref{eq:minresprac2} is equivalent to finding $u^\delta \in X^\delta$ that satisfies
\be \label{h1}
0=\langle R^\delta (G u^\delta-f),R^\delta G w\rangle_{\new{Y}}=(G u^\delta -f)(R^\delta G w) \quad (w \in X^\delta).
\ee

Splitting $Y^\delta$ into the test space $\ran R^\delta G|_{X^\delta}$ and its orthogonal complement, 
one infers that for any $y$ in the latter space and $z \in X^\delta$, it holds that $(G z)(y)=0$, so that
$\sup_{0 \neq y \in Y^\delta} \frac{|(G z)(y)|}{\|y\|_{Y}}=\sup_{0 \neq y \in \ran R^\delta G|_{X^\delta}} \frac{|(G z)(y)|}{\|y\|_{Y}}$, and thus that the value of $\gamma^\delta$ does not change when the space $Y^\delta$ in its definition is replaced by $R^\delta G|_{X^\delta}$.

Using that with $X$ being equipped with $\nrm\cdot\nrm_X$, $G \in \Lis(X,Y')$ is an isometry, an application of \cite[Remark~3.2]{249.99} or \cite[Sect.~2.1]{258.4} concerning Petrov-Galerkin discretisations shows \eqref{eq:gamma}. The final result follows easily.
\end{proof}

\new{\begin{remark} \label{rem:gal-orth} Because the first equality in \eqref{h1} gives $\langle R^\delta G (u-u^\delta),R^\delta G X^\delta \rangle_Y=0$, in particular it holds that
$\sup_{0 \neq y \in Y^\delta} \frac{|(G u^\delta)(y)|}{\|y\|_Y} \leq \sup_{0 \neq y \in Y^\delta} \frac{|(G u)(y)|}{\|y\|_Y}$, which will be used later.
\end{remark}}

\subsection{Saddle-point formulation}
Considering \eqref{eq:minres}, notice that the solution $u \in X$ of $Gu=f$ is equivalently given as
$$
u:=\argmin_{z \in X} \tfrac12 \big(\sum_{i=1}^k \|G_i z-f_i\|_{Y_i'}^2+\sum_{i=k+1}^m \|G_i z-f_i\|_{Y_i}^2\big).
$$
This $u$ solves the Euler-Lagrange equations
$$
\sum_{i=1}^k \langle f_i-G_i u,G_i \undertilde{u}\rangle_{Y_i'}+\sum_{i=k+1}^m \langle f_i-G_i u,G_i \undertilde{u}\rangle_{Y_i}=0\quad (\undertilde{u} \in X).
$$
For $1 \leq i \leq k$ we set $\lambda_i:=R_i(f_i-G_i u)$. Using that $\langle g,h\rangle_{Y_i'}=\langle R_i g,R_i h\rangle_{Y_i}$, we arrive at  the equivalent problem of finding $(\lambda_1,\ldots,\lambda_k,u) \in Y_1\times \cdots\times Y_k \times X$ that solves
$$
\begin{alignedat}{3}
\sum_{i=1}^k \langle \lambda_i, \undertilde{\lambda}_i\rangle_{Y_i} &+\sum_{i=1}^k (G_i u)(\undertilde{\lambda}_i)  & &=\sum_{i=1}^k f_i(\undertilde{\lambda}_i) &\hspace*{-2em} ((\undertilde{\lambda}_1,\ldots,\undertilde{\lambda}_k) \in Y_1\times\cdots\times Y_k)&,\\
\sum_{i=1}^k (G_i \undertilde{u})(\lambda_i)&-\sum_{i=k+1}^m \langle G_i u,G_i \undertilde{u} \rangle_{Y_i}  & &= -\sum_{i=k+1}^m \langle f_i, G_i \undertilde{u}\rangle_{Y_i}
& (\undertilde{u} \in X)&.
\end{alignedat}
$$

Completely analogously, the MINRES solution $u^\delta \in X^\delta$ of \eqref{eq:minresprac} is the last component of the solution 
$(\lambda^\delta_1,\ldots,\lambda^\delta_k,u^\delta) \in Y^\delta_1\times \cdots\times Y^\delta_k\times X^\delta$ that solves the finite dimensional saddle-point
\be \label{eq:saddle}
\begin{split}
& \sum_{i=1}^k \langle \lambda^\delta_i, \undertilde{\lambda}_i\rangle_{Y_i} +\sum_{i=1}^k (G_i u^\delta)(\undertilde{\lambda}_i)  =\sum_{i=1}^k f_i(\undertilde{\lambda}_i) \quad ((\undertilde{\lambda}_1,\ldots,\undertilde{\lambda}_k) \in Y^\delta_1\times\cdots\times Y^\delta_k),\\
& \sum_{i=1}^k (G_i \undertilde{u})(\lambda^\delta_i)\!-\!\!\sum_{i=k+1}^m \!\langle G_i u^\delta,G_i \undertilde{u} \rangle_{Y_i}  = -\sum_{i=k+1}^m \langle f_i, G_i \undertilde{u}\rangle_{Y_i} \quad (\undertilde{u} \in X^\delta).
\end{split}
\ee
Solving this saddle-point can provide a way to determine $u^\delta$ computationally.

\subsection{Reduction to a symmetric positive definite system}
It may however happen that one or more scalar products  $\langle \cdot,\cdot\rangle_{Y_i}$ on the finite dimensional subspaces $Y_i^\delta$ for $1 \leq i \leq k$ are not (efficiently) evaluable, as when $Y_i$ is a fractional Sobolev space. Even when all these scalar products are evaluable, solving a saddle point problem as \eqref{eq:saddle}
is more costly than solving a symmetric positive definite system as with a usual `practical' MINRES  discretisation, where typically all residual components are measured in $L_2$-norms.

Therefore, for $1 \leq i \leq k$,  let $K_i^\delta={K_i^\delta}' \in \Lis({Y_i^\delta}',Y_i^\delta)$ be an operator whose application can be computed efficiently.
Such an operator could be called a preconditioner for $A^\delta_i \in \Lis(Y_i^\delta,{Y_i^\delta}')$ defined by $(A^\delta_i v)(\undertilde{v})=\langle v,\undertilde{v}\rangle_{Y_i}$.
We use $K_i^\delta$ to define the following alternative scalar product on $Y_i^\delta$,
$$
\langle v,\undertilde{v}\rangle_{Y_i^\delta}:=((K_i^\delta)^{-1} v)(\undertilde{v}) \quad (v,\undertilde{v} \in Y_i^\delta),
$$
whose corresponding norm $\|\cdot\|_{Y_i^\delta}$ satisfies
\be \label{eq:twosided}
\lambda_{\min}(K_i^\delta A_i^\delta) \|\cdot\|_{Y_i^\delta}^2 \leq \|\cdot\|_{Y_i}^2 \leq \lambda_{\max}(K_i^\delta A_i^\delta) \|\cdot\|_{Y_i^\delta}^2.
\ee

\begin{remark} Given a basis $\Phi_i$ for $Y_i^\delta$, with $\cF_i:= g\mapsto g(\Phi_i) \in \Lis({Y_i^\delta}',\R^{\# \Phi_i})$, and so $\cF_i'\colon {\bf w}\mapsto {\bf w}^\top \Phi_i \in \Lis(\R^{\# \Phi_i},Y_i^\delta)$, ${\bf A}_i:=\cF_i A_i^\delta \cF_i'$ is known as a stiffness matrix.
Given some symmetric positive definite ${\bf K}_i \eqsim {\bf A}_i^{-1}$, which is more appropriately called a preconditioner, 
setting $K_i^\delta:=\cF_i' {\bf K}_i \cF_i$ gives $\sigma(K_i^\delta A_i^\delta)=\sigma({\bf K}_i{\bf A}_i)$.
\end{remark}

We now \emph{replace} \eqref{eq:minresprac} by 
\be \label{eq:minrespracnew}
u^\delta:=\argmin_{z \in X^\delta} \tfrac12 \big(\sum_{i=1}^k \sup_{0 \neq y_i \in Y_i^\delta} \frac{|(G_i z-f_i)(y_i)|^2}{\|y_i\|_{Y^\delta_i}^2}+
\sum_{i=k+1}^m \|G_i z-f_i\|_{Y_i}^2\big),
\ee
which is a  \emph{fully practical} MINRES discretisation.
Indeed by making the corresponding replacement of $\langle \lambda_i^\delta, \undertilde{\lambda}_i\rangle_{Y_i}$ by $((K_i^\delta)^{-1} \lambda_i^\delta)(\undertilde{\lambda}_i)$ in \eqref{eq:saddle}, and subsequently eliminating $\lambda_1^\delta,\ldots,\lambda_k^\delta$ from the resulting system, one infers that this latter $u^\delta$ can be computed as the solution in $X^\delta$ of the symmetric positive definite system
\be \label{eq:minresfullyprac}
\sum_{i=1}^k(G_i \undertilde{u})(K_i^\delta(G_i u^\delta-f_i))+\sum_{i=k+1}^m\langle G_i\undertilde{u},G_i u^\delta -f_i\rangle_{Y_i}=0 \quad (\undertilde{u} \in X^\delta).
\ee

\begin{theorem} \label{thm:quasi-opt2}
Let $\gamma^\delta>0$. Then with $M^\delta:=\max\big(1,\max_{1 \leq i \leq k} \lambda_{\max}(K_i^\delta A_i^\delta)\big)$, $m^\delta:=\min\big(1,\min_{1 \leq i \leq k} \lambda_{\min}(K_i^\delta A_i^\delta)\big)$,
$u^\delta$ from \eqref{eq:minrespracnew} satisfies
\new{$$
\nrm u-u^\delta \nrm_X \leq \tfrac{\sqrt{M^\delta}}{ \gamma^\delta \sqrt{m^\delta}}
\inf_{w \in X^\delta}\nrm u-w\nrm_X,
$$}
and so
$$
\|u-u^\delta\|_X \leq \new{\sqrt{\tfrac{M^\delta}{m^\delta}}}
\tfrac{\|G\|_{\cL(X,Y')}\|G^{-1}\|_{\cL(Y',X)}}{\gamma^\delta}\inf_{w \in X^\delta}\| u-w\|_X.
$$
\end{theorem}

\begin{proof} \new{When we equip
 $Y^\delta$ with $\|\cdot\|_{Y^\delta}:=\sqrt{\sum_{i=1}^k \|\cdot\|_{Y_i^\delta}^2+\sum_{i=k+1}^m \|\cdot\|_{Y_i}^2}$ instead of with $\|\cdot\|_Y$, the MINRES solution $u^\delta$ from \eqref{eq:minresprac} is of the form of the MINRES solution from \eqref{eq:minrespracnew}. It holds that
$$
m^\delta \|\cdot\|_{Y^\delta}^2 \leq \|\cdot\|_{Y}^2 \leq M^\delta\|\cdot\|_{Y^\delta}^2.
$$

The mapping $S^\delta:=u=G^{-1}f \mapsto u^\delta$ is a projector onto $X^\delta$. Since it suffices to consider the case that $\{0\}\subsetneq X^\delta \subsetneq X$, we have
$$
\nrm u-u^\delta\nrm_X \leq \sup_{0 \neq z \in X}\frac{\nrm S^\delta z\nrm_X}{\nrm z\nrm_X} \inf_{w \in X^\delta}\nrm u-w\nrm_X.
$$
Because of the replacement of $\|\cdot\|_Y$ by $\|\cdot\|_{Y^\delta}$ on $Y^\delta$, the estimate derived in Remark~\ref{rem:gal-orth} now reads as
$$
\sup_{0 \neq y \in Y^\delta} \frac{|(G S^\delta z)(y)|}{\|y\|_{Y^\delta}} \leq \sup_{0 \neq y \in Y^\delta} \frac{|(G z)(y)|}{\|y\|_{Y^\delta}} \quad (z \in X).
$$
For $w \in X^\delta$, it holds that
$$
\|G w\|_{Y'} \leq \tfrac{1}{\gamma^\delta} \sup_{0 \neq y \in Y^\delta} \frac{|(G w)(y)|}{\|y\|_Y} \leq \tfrac{1}{\gamma^\delta \sqrt{m^\delta}} \sup_{0 \neq y \in Y^\delta} \frac{|(G w)(y)|}{\|y\|_{Y^\delta}}.
$$
We conclude that for  $z \in X$,
\begin{align*}
\nrm S^\delta z\nrm_X & \leq  \tfrac{1}{\gamma^\delta \sqrt{m^\delta}} \sup_{0 \neq y \in Y^\delta} \frac{|(G S^\delta z)(y)|}{\|y\|_{Y^\delta}}
\leq  \tfrac{1}{\gamma^\delta \sqrt{m^\delta}} \sup_{0 \neq y \in Y^\delta} \frac{|(G z)(y)|}{\|y\|_{Y^\delta}}\\
& \leq  \tfrac{ \sqrt{M^\delta}}{\gamma^\delta \sqrt{m^\delta}} \sup_{0 \neq y \in Y^\delta} \frac{|(G z)(y)|}{\|y\|_Y}
\leq  \tfrac{ \sqrt{M^\delta}}{\gamma^\delta \sqrt{m^\delta}} \nrm z\nrm_X,
\end{align*}
which completes the proof.}
\end{proof}

Notice that Theorem~\ref{thm:quasi-opt2} generalizes \eqref{eq:upp} from Theorem~\ref{thm:quasi-opt} (indeed, take $K_i^\delta=(A_i^\delta)^{-1}$), which in turn generalized \eqref{eq:quasi-best} (take $Y_i^\delta=Y_i$).

The bilinear form $(w,\tilde{w}) \mapsto
\sum_{i=1}^k(G_i \undertilde{w})(K_i^\delta G_i w)+\sum_{i=k+1}^m\langle G_i\undertilde{w},G_i w\rangle_{Y_i}$ on $X \times X$ is symmetric, bounded (with constant $M^\delta \|G\|_{\cL(X,Y')}^2$), and, restricted to $X^\delta \times X^\delta$, coercive (with constant $\new{m^\delta} \|G^{-1}\|_{\cL(Y',X)}^{-2} (\gamma^\delta)^2 $).
The way to solve \eqref{eq:minresfullyprac} is by the application of the preconditioned conjugate gradient method, for some self-adjoint preconditioner in $\Lis({X^\delta}',X^\delta)$.

\subsection{Fortin interpolators and a posteriori error estimation}
As is well known, validity of the inf-sup condition $\gamma_i^\delta>0$ in \eqref{eq:defgamma} is equivalent to existence of a Fortin interpolator.
The following formulation from {\cite[Prop.~5.1]{249.992}} gives a precise quantitative statement, whereas it does not require injectivity of $G_i$ which is not guaranteed in our applications.

\begin{theorem} \label{thm:Fortin} Let $G_i \in \cL(X,Y_i')$. Assuming $G_i X^\delta \neq \{0\}$ and $Y_i^\delta \neq \{0\}$, let
\be \label{fortin}
\Pi_i^\delta \in \cL(Y_i,Y_i^\delta) \text{ with } (G_i X^\delta)\big((\identity - \Pi_i^\delta)Y_i\big)=0.
\ee
Then $\gamma_i^\delta \geq \|\Pi_i^\delta\|_{\cL(Y_i,Y_i)}^{-1}$.

Conversely, when $\gamma_i^\delta >0$, then there exists a $\Pi_i^\delta$ as in \eqref{fortin}, being even a projector onto $Y_i^\delta$, with
$\|\Pi_i^\delta\|^{-1}_{\cL(Y_i,Y_i)} = \gamma_i^\delta$.
\end{theorem}

As mentioned in the introduction, an advantage of a MINRES discretisation is that the norm of the residual is an efficient and reliable a posteriori estimator of the norm of the error.
In the setting \eqref{eq:minres}, where $G \in \Lis(X,V)$ with $V=Y'_1\times\cdots\times Y'_k\times Y_{k+1}\times\cdots\times Y_m$, and so, when $k>0$, one or more components of the residual are measured in dual norms, this a posteriori estimator is not computable.
To arrive at a practical MINRES discretisation, we have replaced these dual norms by computable discretised dual norms, and nevertheless ended up with having quasi-optimal approximations (see Theorem~\ref{thm:quasi-opt2}). 
When it comes to a posteriori error estimation, however, there is some price to be paid. As we will see below, our computable posteriori estimator will only be reliable modulo a data-oscillation term.
A similar analysis in the context of DPG methods can already be found in \cite{35.93556}.

Let $w \in X^\delta$. Then
\be \label{e1}
\|u-w\|_X \in  \big[\|G\|_{\cL(X,V)}^{-1} \|f-Gw\|_V,\|G^{-1}\|_{\cL(X,V)} \|f-Gw\|_V\big],
\ee
where
$$
 \|f-Gw\|_V^2=\sum_{i=1}^k\|f_i-G_iw\|_{Y'_i}^2+\sum_{i=k+1}^m\|f_i-G_iw\|_{Y_i}^2.
$$
For $1\leq i \leq k$, let $\Pi_i^\delta$ be a valid Fortin interpolator.
Then for $\tilde{y}_i \in Y_i$,
\be  \label{e3}
\begin{split}
& |(f_i-G_i w)(\tilde{y}_i)|  \leq |(f_i-G_i w)(\Pi_i^\delta \tilde{y}_i)|+|f_i((\identity-\Pi_i^\delta)\tilde{y}_i)|\\
& \leq \|\Pi_i^\delta \tilde{y}_i\|_{Y_i^\delta} \sup_{0 \neq y_i \in Y_i^\delta}\frac{|(f_i-G_i w)(y_i)|}{\|y_i\|_{Y_i^\delta}}+\|(\identity-{\Pi_i^\delta}')f_i\|_{Y_i'} \|\tilde{y}_i\|_{Y_i}\\
& \leq \Big(\|\Pi_i^\delta\|_{\cL(Y_i,Y_i)} \lambda_{\min}(K_i^\delta A_i^\delta)^{-\frac12} \sup_{0 \neq y_i \in Y_i^\delta}\frac{|(f_i-G_i w)(y_i)|}{\|y_i\|_{Y_i^\delta}}+\|(\identity-{\Pi_i^\delta}')f_i\|_{Y_i'}\Big)\|\tilde{y}_i\|_{Y_i}.
\end{split}
\ee
From \eqref{e1}-\eqref{e3} one easily infers the upper bound for $\|u-w\|_X^2$ given in the following proposition, whereas the derivation of the lower bound  is easier.

\begin{proposition} For $w \in X^\delta$, the computable (squared) estimator
$$
\cE^\delta(w,f)^2:=\sum_{i=1}^k \sup_{0 \neq y_i \in Y_i^\delta}\frac{|(f_i-G_i w)(y_i)|^2}{\|y_i\|^2_{Y_i^\delta}}+ \sum_{i=k+1}^m\|f_i-G_iw\|_{Y_i}^2
$$
satisfies
\begin{align*}
\|G\|_{\cL(X,V)}^{-2}& \min\big(1,\min_{1 \leq i \leq k} \lambda_{\max}(K_i^\delta A_i^\delta)^{-1}\big) \cE^\delta(w,f)^2
\leq \|u-w\|_X^2 \leq \\
&\|G^{-1}\|_{\cL(V,X)}^2 \max\big(1,2 \max_{1 \leq i \leq k} \lambda_{\min}(K_i^\delta A_i^\delta)^{-1} \|\Pi_i^\delta\|_{\cL(Y_i,Y_i)}^2 \big)  \cE^\delta(w,f)^2  \\
&+ 2\|G^{-1}\|_{\cL(V,X)}^2\sum_{i=1}^k \|(\identity-{\Pi_i^\delta}')f_i\|^2_{Y_i'}.
\end{align*}
\end{proposition}

\begin{remark}[Bounding the oscillation term] \label{rem2}
By taking $\Pi_i^\delta$ being the Fortin projector with $\|\Pi_i^\delta\|_{\cL(Y_i,Y_i)}=1/\gamma_i^\delta$, for $\{0\} \subsetneq Y_i^\delta \subsetneq Y_i$
it holds that
\begin{align*}
& \|(\identity-{\Pi_i^\delta}')f_i\|_{Y_i'}=\sup_{0 \neq y_i \in Y_i}\frac{|f_i((\identity-\Pi_i^\delta)y_i)|}{\|y_i\|_{Y_i}}=\\
& \sup_{0 \neq y_i \in Y_i} \inf_{0 \neq w \in X^\delta} \frac{|G_i(u-w)((\identity-\Pi_i^\delta)y_i)|}{\|y_i\|_{Y_i}}
 \leq \tfrac{1}{\gamma_i^\delta} \|G_i\|_{\cL(X,Y_i')}\inf_{0 \neq w \in X^\delta}\|u-w\|_X,
 \end{align*}
 and so
 $$
 \osc^\delta(f):=\sqrt{\sum_{i=1}^k \|(\identity-\Pi^\delta_i)'f_i\|^2_{Y_i'}} \leq \|G\|_{\cL(X,V)} \sqrt{\sum_{i=1}^k  \tfrac{1}{{\gamma_i^\delta}^2} } \inf_{0 \neq w \in X^\delta}\|u-w\|_X.
$$
In other words, the data-oscillation is bounded by a multiple of the best approximation error.

It would be even better when, for $1 \leq i \leq k$, $Y_i^\delta$ is chosen such that it allows for the construction of a (uniformly bounded) Fortin interpolator $\Pi^\delta_i$ such that, for general, sufficiently smooth $u$ and $f$, $\osc^\delta(f)$ is of higher order than $ \inf_{0 \neq w \in X^\delta}\|u-w\|_X$, so that besides being an efficient estimator
one can expect that in any case asymptotically $\cE^\delta(w,f)$ is also a reliable one.
\end{remark}

\begin{remark}[Computing $\cE^\delta(u^\delta,f)$] \label{rem:comp}
If $w=u^\delta$ is the MINRES solution from \eqref{eq:saddle}, then the term $\sup_{0 \neq y_i \in Y_i^\delta}\frac{|(f_i-G_i u^\delta)(y_i)|^2}{\|y_i\|^2_{Y_i^\delta}}$ in the expression for  $\cE^\delta(u^\delta,f)^2$ is equal to $\|\lambda_i^\delta\|_{Y_i}^2$.

If $w=u^\delta$ is the MINRES solution from the symmetric positive definite system \eqref{eq:minresfullyprac}, then $\sup_{0 \neq y_i \in Y_i^\delta}\frac{|(f_i-G_i u^\delta)(y_i)|^2}{\|y_i\|^2_{Y_i^\delta}}$ is equal to $(G_i u^\delta-f_i)(K_i^\delta(G_i u^\delta-f_i))$.
\end{remark}

\section{Verification of the inf-sup conditions} \label{sec:infsup}

By constructing Fortin interpolators $\Pi_i$ for the MINRES examples from Sect.~\ref{sec:examples}, we verify the inf-sup conditions $\gamma_i>0$, 
which, for finite element spaces of given fixed orders, will hold uniformly over
uniformly shape regular, possibly locally refined partitions.

If $(G_i X^\delta)\big((\identity - \Pi_i^\delta)Y_i\big)=0$, then this obviously also holds when $X^\delta$ is replaced by a subspace.
Consequently, for Examples~\ref{ex:2ndorder}, \ref{ex:mild}, and \ref{ex:mild-weak}, it suffices to consider Case (ii).

\subsection{Inf-sup conditions for Example~\ref{ex:2ndorder}\ref{ex:2ndorder2} (2nd order formulation)} \label{sec:4.1}
We assume that $\Omega \subset \R^d$ is a polytope, and let $\tria^\delta$ be a conforming, shape regular partition of $\Omega$ into (closed) $d$-simplices.
With $\cF(\tria^\delta)$ we denote the set of (closed) facets of $K \in  \tria^\delta$.
We assume that $\Gamma_D$ is the union of some $e \in \cF(\tria^\delta)$.
For $K \in \tria^\delta$, we set the patches $\omega_{K,0}(\tria^\delta):=K$, and $\omega_{K,i+1}(\tria^\delta):=\cup\{K' \in \tria^\delta\colon K' \cap \omega_{K,i}(\tria^\delta) \neq \emptyset\}$. Let $h_\delta$ be the piecewise constant function on $\Omega$ defined by $h_\delta|_K:=|K|^{1/d}$.
Focussing on the case of having inhomogeneous Dirichlet boundary conditions on $\Gamma_D$, i.e., Ex.~\ref{ex:2ndorder}\ref{ex:2ndorder2}, we take
\be \label{g1}
X^\delta=\cS_p^0(\tria^\delta):=\cS_p^{-1}(\tria^\delta) \cap C(\Omega),
\ee
with $\cS_p^{-1}(\tria^\delta) $ being the space of $f\colon \Omega \rightarrow \R$ such that for $K \in \tria^\delta$, $f|_K \in \cP_{p}(K)$, being the space of 
polynomials of maximal degree $p$.

We take $A=\identity$, although the arguments given below apply equally when $A$ is piecewise constant w.r.t.~$\tria^\delta$.
For convenience, we take $B=0$, but the case of $B$ being a PDO of first order with piecewise constant coefficients w.r.t.~$\tria^\delta$ poses no additional difficulties.\footnote{It suffices to take $Y_1^\delta:=\cS^0_{p+d+1}(\tria^\delta) \cap H^1_{0,\Gamma_D}(\Omega)$}

Considering the original `impractical' MINRES discretisation \eqref{eq:2ndorder}, as discussed before we write the term $\|\gamma_D w -h_D\|_{H^{\frac12}(\Gamma_D)}^2$ as
$\|\gamma_D w -h_D\|_{\widetilde{H}^{-\frac12}(\Gamma_D)'}^2$.
For constructing a MINRES discretisation of type \eqref{eq:minresprac} that is quasi-optimal, it therefore
 suffices to select finite dimensional subspaces
 $$
 Y_1^\delta \subset  Y_1=H^1_{0,\Gamma_D}(\Omega),\quad
Y_2^\delta  \subset Y_2=\widetilde{H}^{-\frac12}(\Gamma_D)
$$
that allow for the construction of Fortin interpolators 
$\Pi_1^\delta \in \cL(H^1_{0,\Gamma_D}(\Omega),Y_1^\delta)$, 
and $\Pi_2^\delta \in \cL(\widetilde{H}^{-\frac12}(\Gamma_D),Y_2^\delta)$ with
\begin{alignat}{2}
\label{f1}
&\int_\Omega \nabla w \cdot\nabla (\identity -\Pi_1^\delta) v \,dx=0 \quad&&(w \in X^\delta,\,v \in H^1_{0,\Gamma_D}(\Omega)),\\ \label{f2}
& \int_{\Gamma_D} w (\identity -\Pi_2^\delta) v\,ds=0  && (w \in X^\delta,\,v \in \widetilde{H}^{-\frac12}(\Gamma_D)).
\end{alignat}

Starting with \eqref{f1}, we rewrite it as
$$
0=\sum_{K \in \tria^\delta} \big\{\int_K -\Delta w (\identity -\Pi_1^\delta) v\,dx+\int_{\partial K} \frac{\partial w}{\partial \vec{n}} (\identity -\Pi_1^\delta) v\,ds\big\}
\quad (w \in X^\delta,\,v \in H^1_{0,\Gamma_D}(\Omega)), 
$$
and select
\be \label{g2}
Y_1^\delta:=\cS^0_{p+d-1}(\tria^\delta) \cap H^1_{0,\Gamma_D}(\Omega).
\ee
It suffices to construct  $\Pi_1^\delta \in \cL(H^1_{0,\Gamma_D}(\Omega),Y_1^\delta)$ such that both
\begin{align} \label{100}
\ran(\identity-\Pi_1^\delta)|_e \perp_{L_2(e)} \cP_{p-1}(e) \quad(e \in \cF(\tria^\delta)),
\intertext{and, when $p>1$,} \label{101}
\ran(\identity-\Pi_1^\delta)|_K \perp_{L_2(K)} \cP_{p-2}(K) \quad(K \in \tria^\delta).
\end{align}

Let $\hat{\Pi}_1^\delta\colon H^1_{0,\Gamma_D}(\Omega) \rightarrow S^0_1(\tria^\delta) \cap H^1_{0,\Gamma_D}(\Omega)$ denote the familiar Scott-Zhang interpolator (\cite{247.2}). It satisfies 
$$
\|h_\delta^{-1}(\identity-\hat{\Pi}_1^\delta)v\|_{L_2(K)} + |\hat{\Pi}_1^\delta v|_{H^1(K)} \lesssim |v|_{H^1(\omega_{K,1}(\tria^\delta))} \quad(v \in H^1_{0,\Gamma_D}(\Omega) ).
$$
In two steps we correct $\hat{\Pi}_1^\delta$ to a $\Pi_1^\delta \in \cL(H^1_{0,\Gamma_D}(\Omega),Y_1^\delta)$ that satisfies \eqref{100}-\eqref{101}.

On a facet $\hat{e}$ of a reference $d$-simplex $\hat{K}$, let $b_{\hat{e}}$ denote the $d$-fold product of its barycentric coordinates.
From $\int_{\hat{e}} b_{\hat{e}} |q|^2\,ds \eqsim \int_{\hat{e}} |q|^2\,ds$ ($q \in \cP_{p-1}(\hat{e})$), and $b_{\hat{e}} \cP_{p-1}(\hat{e})=\cP_{p+d-1}(\hat{e}) \cap H^1_0(\hat{e})$, one infers that there exist bases $\{\hat{\tilde{\psi}}_i\}_i$ and $\{\hat{\ell}_i\}_i$ of $\cP_{p+d-1}(\hat{e}) \cap H^1_0(\hat{e})$ and $\cP_{p-1}(\hat{e})$ that are $L_2(\hat{e})$-biorthogonal. Let $\hat{\psi}_i$ be an extension of $\hat{\tilde{\psi}}_i$ to a function in $\cP_{p+d-1}(\hat{K}) \cap H^1_{0,\partial \hat{K}\setminus {\rm int}(\hat{e})}(\hat{K})$.

By using affine bijections between $\hat{K}$ and $K \in \tria^\delta$, for each $e \in \cF(\tria^\delta)$ we lift $\{\hat{\ell}_i\}_i$ to a collection $\{\ell_{e,i}\}_i$ that spans $\cP_{p-1}(e)$, and  lift $\{\hat{\psi}_i\}_i$ to a collection $\{\psi_{e,i}\}_i \subset Y_1^\delta$ of functions supported on the union of the two (or one) simplices in $\tria^\delta$ of which $e$ is a facet.
We set 
$$
\breve{\Pi}_1^\delta v:=\hat{\Pi}_1^\delta v+\sum_{e \in  \cF(\tria^\delta)}\sum_i \frac{\langle v-\hat{\Pi}_1^\delta v,\ell_{e,i}\rangle_{L_2(e)}}{\langle \psi_{e,i},\ell_{e,i}\rangle_{L_2(e)}}\psi_{e,i}
\quad (v \in H^1_{0,\Gamma_D}(\Omega)).
$$
From $\langle \psi_{e,i},\ell_{e,j}\rangle_{L_2(e)}=0$ when $i \neq j$, it follows that
\be \label{12}
\ran (\identity - \breve{\Pi}_1^\delta)|_{e} \perp_{L_2(e)} \cP_{p-1}(e) \quad \quad(e \in \cF(\tria^\delta)).
\ee
Standard homogeneity arguments and the use of the trace inequality show that 
$$
\|h_\delta^{-1}(\identity-\breve{\Pi}_1^\delta)v\|_{L_2(K)} + |\breve{\Pi}_1^\delta v|_{H^1(K)} \lesssim |v|_{H^1(\omega_{K,2}(\tria^\delta))} \quad(v \in H^1_{0,\Gamma_D}(\Omega) ).
$$

For the case that $p=1$, we take $\Pi_1^\delta= \breve{\Pi}_1^\delta$. Otherwise we proceed as follows.
Let $b_{\hat{K}}$ denote the $(d+1)$-fold product of the barycentric coordinates of $\hat{K}$.
From $\int_{\hat{K}} b_{\hat{K}} |q|^2\,dz \eqsim \int_{\hat{K}} |q|^2\,dx$ ($q \in \cP_{p-2}(\hat{K})$), and $b_{\hat{K}} \cP_p(\hat{K})=\cP_{p+d-1}(\hat{K}) \cap H^1_0(\hat{K})$, one infers that there exist bases $\{\hat{\phi}_k\}_k$ and $\{\hat{q}_k\}_k$ of $\cP_{p+d-1}(\hat{K}) \cap H^1_0(\hat{K})$ and $\cP_{p-2}(\hat{K})$ that are $L_2(\hat{K})$-biorthogonal. 

Again using the affine bijections between $\hat{K}$ and $K \in \tria^\delta$, for each $K \in \tria^\delta$ 
we lift $\{\hat{\phi}_k\}_k$ and $\{\hat{q}_k\}_k$ to collections $\{\phi_{K,k}\}_k$ and $\{q_{K,k}\}_k$ that span $\cP_{p+d-1}(K) \cap H^1_0(K)$ and $\cP_{p-2}(K)$, respectively.
We set 
$$
\Pi_1^\delta v:=\breve{\Pi}_1^\delta v+\sum_{K \in  \tria^\delta}\sum_k \frac{\langle v-\breve{\Pi}_1^\delta v,q_{K,k}\rangle_{L_2(K)}}{\langle \phi_{K,k},q_{K,k}\rangle_{L_2(K)}}\phi_{K,k}
\quad (v \in H^1_{0,\Gamma_D}(\Omega)).
$$
Thanks to \eqref{12}, it satisfies \eqref{100}, and from $\langle \phi_{K,k},q_{K,k'}\rangle_{L_2(K)}=0$ when $k \neq k'$, one infers that it satisfies \eqref{101}.
From 
$$
\|h_\delta^{-1}(\identity-\Pi_1^\delta)v\|_{L_2(K)} +  |\Pi_1^\delta v|_{H^1(K)} \lesssim |v|_{H^1(\omega_{K,2}(\tria^\delta))} \quad(v \in H^1_{0,\Gamma_D}(\Omega) ),
$$
we conclude the following result.
\newcounter{footnoteValueSaver}

\begin{proposition} \label{prop1} For $X^\delta$ and $Y_1^\delta$ from \eqref{g1} and \eqref{g2}, it holds that $\Pi_1^\delta \in \cL(H^1_{0,\Gamma_D}(\Omega),Y_1^\delta)$,\footnote{\emph{Uniformly} in all $\tria^\delta$ that satisfy a uniform shape regularity condition.} and \eqref{f1} is valid.
\end{proposition}
\setcounter{footnoteValueSaver}{\value{footnote}}

In view of a posteriori error estimation, we consider the data-oscillation term associated to $\Pi_1^\delta$ (actually a slightly modified operator).
We show that it is of higher order than $\inf_{w \in X^\delta}\|u-w\|_{H^1(\Omega)}$ (cf.~Remark~\ref{rem2}) when we take the larger space $Y_1^\delta= \cS^0_{p+d}(\tria^\delta) \cap H^1_{0,\Gamma_D}(\Omega)$.

\begin{remark}[data-oscillation] \label{rem1}
With $\breve{P}_1^\delta:=v \mapsto \sum_{e \in  \cF(\tria^\delta)}\sum_i \frac{\langle v,\ell_{e,i}\rangle_{L_2(e)}}{\langle \psi_{e,i},\ell_{e,i}\rangle_{L_2(e)}}\psi_{e,i}$,
and $P_1^\delta:=v \mapsto  \sum_{K \in  \tria^\delta}\sum_k \frac{\langle v,q_{K,k}\rangle_{L_2(K)}}{\langle \phi_{K,k},q_{K,k}\rangle_{L_2(K)}}\phi_{K,k}$,
it holds that
$\breve{\Pi}_1^\delta=\hat{\Pi}_1^\delta+\breve{P}_1^\delta(\identity-\hat{\Pi}_1^\delta)$, and
$\Pi_1^\delta=\breve{\Pi}_1^\delta+P_1^\delta(\identity-\breve{\Pi}_1^\delta)$,
so that $\identity-\Pi_1^\delta=(\identity-P_1^\delta)(\identity-\breve{P}_1^\delta)(\identity-\hat{\Pi}_1^\delta)$, and so
$$
\identity-{\Pi_1^\delta}'= (\identity-{\hat{\Pi}_1^\delta}')(\identity-{\breve{P}_1^\delta}')(\identity-{P_1^\delta}').
$$
We now replace the Scott-Zhang interpolator $\hat{\Pi}_1^\delta$ by the interpolator onto $S^0_1(\tria^\delta) \cap H^1_{0,\Gamma_D}(\Omega)$ from  \cite{250.1,64.586}\new{,} which does not affect the validity of Proposition~\ref{prop1}.
This new $\hat{\Pi}_1^\delta$ additionally satisfies 
$\|(\identity-{\hat{\Pi}_1^\delta}') f_1\|_{Y_1'} \lesssim \|h_\delta f_1\|_{L_2(\Omega)}$ ($f_1 \in L_2(\Omega)$).
By using this estimate together with the stability and locality of $\breve{P}_1^\delta$ and $P_1^\delta$, and the fact that ${P_1^\delta}'$ reproduces $\cS^{-1}_{p-1}(\tria^\delta)$ (instead of $\cS^{-1}_{p-2}(\tria^\delta)$ for $Y_1^\delta= \cS^0_{p+d-1}(\tria^\delta) \cap H^1_{0,\Gamma_D}(\Omega)$), one infers that
$$
\|(\identity-{\Pi_1^\delta}')f_1\|_{H^1_{0,\Gamma_D}(\Omega)'} \lesssim \sqrt{\sum_{K \in \tria^\delta} (h_\delta|_K)^{2(p+1)}|f_1|_{H^{p-1}(K)}^2} \quad (f_1 \in H^{\new{p}}(\Omega)).\qedhere
$$
\end{remark}

To construct the Fortin interpolator $\Pi_2^\delta$, with $\cF^\delta_{\Gamma_D}:=\{e \in \cF(\tria^\delta)\colon e \subset \Gamma_D\}$ we take
\be \label{g3}
Y^\delta_2:=\cS^{-1}_{p}(\cF^\delta_{\Gamma_D}).
\ee
With $\{\phi^\delta_i\}$ being the nodal basis of $\cS_p^0(\cF^\delta_{\Gamma_D}) \supset \ran \Gamma_D|_{X^\delta}$,
it is known that a projector $P_2^\delta$ of Scott-Zhang type exists of the form $P_2^\delta v =\sum_i \langle v,\psi^\delta_i\rangle_{L_2(\Gamma_D)} \phi^\delta_i$,
where $\{\psi^\delta_i\} \subset Y^\delta_2$ is biorthogonal to $\{\phi^\delta_i\}$,
$P_2^\delta$ is bounded in $L_2(\Gamma_D)$ and in $H^1(\Gamma_D)$, and 
\be \label{103}
\|(\identity-P_2^\delta)f_2\|_{H^{\frac12}(\Gamma_D)} \lesssim \sqrt{\sum_{e \in \cF^\delta_{\Gamma_D}} (h_\delta|_e)^{2p+1}|f_2|_{H^{p+1}(e)}^2} \quad (f_2 \in H^{p+1}(\Omega)).
\ee
Since $\Pi_2^\delta:={P_2^\delta}'$ maps into $Y_2^\delta$, and $P_2^\delta$ reproduces $\cS_p^0(\cF^\delta_{\Gamma_D})$, we conclude the following result.

\begin{proposition} \label{prop2} For $X^\delta$ and $Y_2^\delta$ from \eqref{g1} and \eqref{g3}, it holds that $\Pi_2^\delta \in \cL(\widetilde{H}^{-\frac12}(\Gamma_D),Y_2^\delta)$,\footnotemark[\value{footnoteValueSaver}] and \eqref{f2} is valid.
\end{proposition}

\begin{remark}[data-oscillation]
Equation \eqref{103} shows that the data-oscillation term corresponding to $\Pi_2^\delta$ is of higher order than the best approximation error.
\end{remark}

\subsection{Inf-sup conditions for Example~\ref{ex:mild}\ref{mild2} (mild formulation)}
We take
\be \label{g4}
X^\delta:=\RT_{p-1}(\tria^\delta) \times \cS^0_p(\tria^\delta),
\ee
where $\RT_{p-1}(\tria^\delta)=\RT^{-1}_{p-1}(\tria^\delta) \cap H(\divv;\Omega)$ and $\RT^{-1}_{p-1}(\tria^\delta)=\{\vec{q}\in L_2(\Omega)^d\colon \vec{q}|_K \in \cP_{p-1}(K)^d+\vec{x}\cP_{p-1}(K)\}$.
The term $\|\gamma_D w -h_D\|_{H^{\frac12}(\Gamma_D)}^2=\|\gamma_D w -h_D\|_{\widetilde{H}^{-\frac12}(\Gamma_D)'}^2$
can be handled as in Example~\ref{ex:2ndorder}. The dual norm can be discretized by replacing $\widetilde{H}^{-\frac12}(\Gamma_D)$ by $\cS^{-1}_{p}(\cF^\delta_{\Gamma_D})$.

Considering the term $\|\gamma_N \vec{q}-h_N\|_{H^{-\frac12}(\Gamma_N)}^2$,
using that $\ran \gamma_N|_{\RT_{p-1}(\tria^\delta)} = \cS^{-1}_p(\cF_{\Gamma_N}^\delta)$, one needs to select a finite dimensional subspace $Y_1^\delta \subset  Y_1=H_{00}^{\frac12}(\Gamma_N)$ that  allows for the construction of a Fortin interpolator
$\Pi_1^\delta \in \cL(H_{00}^{\frac12}(\Gamma_N),Y_1^\delta)$ with
\be \label{102}
 \int_{\Gamma_N} w (\identity -\Pi_1^\delta) v\,ds=0  \quad (w \in \cS^{-1}_{p-1}(\cF_{\Gamma_N}^\delta) ,\,v \in H_{00}^{\frac12}(\Gamma_N)).
\ee

We take
\be \label{g5}
Y_1^\delta:=\cS^0_{p+d-1}(\cF^\delta_{\Gamma_N}) \cap H^1_0(\Gamma_N),
\ee
and follow a somewhat simplified version of the construction of $\Pi_1^\delta$ in Sect.~\ref{sec:4.1}.
Let $\hat{\Pi}_1^\delta$ be a modified Scott-Zhang projector onto $\cS_1^0(\cF_{\Gamma_N}^\delta)\cap H^1_0(\Gamma_N)$ from \cite{64.586}.
For $e \in \cF_{\Gamma_N}^\delta$, we can find $\{\phi_{e,k}\}$ and $\{q_{e,k}\}$, which up to a scaling are $L_2(e)$-biorthogal, and that span $\cP_{d+p-1}(e) \cap H_0^1(e)$ and $\cP_{p-1}(e)$, respectively, such that for $\Pi_1^\delta$ defined by
$$
\Pi_1^\delta v:=\hat{\Pi}_1^\delta v+\sum_{e \in  \cF_{\Gamma_N}^\delta}\sum_k \frac{\langle v-\hat{\Pi}_1^\delta v,q_{e,k}\rangle_{L_2(e)}}{\langle \phi_{e,k},q_{e,k}\rangle_{L_2(e)}}\phi_{e,k},
$$
the following result is valid.

\begin{proposition} \label{prop3} For $X^\delta$ and $Y_1^\delta$ from \eqref{g4} and \eqref{g5}, it holds that $\Pi_1^\delta \in \cL(H_{00}^{\frac12}(\Gamma_N),Y_1^\delta)$,\footnotemark[\value{footnoteValueSaver}] and \eqref{102} is valid.
\end{proposition}

\begin{remark}[data-oscillation]
It holds that
$$
\|(\identity-{\Pi_1^\delta}')f_1\|_{H^{-\frac12}(\Gamma_N)} \lesssim \sqrt{\sum_{e \in \cF_{\Gamma_N}^\delta} (h_\delta|_K)^{2p+1}|f_1|_{H^{p}(K)}^2} \quad (f_1 \in H^{p}(\Omega)),
$$
so the data-oscillation term corresponding to $\Pi_1^\delta$ is of higher order than the best approximation error.
\end{remark}

\begin{remark}[Avoidance of the condition $g \in L_2(\Omega)$] \label{rem:singulardata} 
Consider the mild formulation with homogeneous boundary data $h_D=0$ and $h_N=0$ (i.e., Example~\ref{ex:mild}\ref{mild1}),
so that $G(\vec{q},w)=(\vec{q}-A u,B w-\divv \vec{q})$.
As  noticed before, a disadvantage of this formulation is that it requires a forcing term $g \in L_2(\Omega)$.
As shown in \cite{75.259, 75.068}, assuming $B=0$ this condition can be circumvented by replacing a general $g \in H^1_{0,\Gamma_D}(\Omega)'$ by a finite element approximation, resulting in a MINRES method that is quasi-optimal in the weaker $L_2(\Omega)^d \times H^1(\Omega)$-norm.
The analysis in \cite{75.068} was restricted to the lowest order case, and below we generalise it to finite element approximation of general degree.

For
$$
X^\delta:=\big(\RT_{p-1}(\tria^\delta) \cap H_{0,\Gamma_N}(\divv;\Omega)\big) \times \big(\cS^0_p(\tria^\delta) \cap H^1_{0,\Gamma_D}(\Omega)\big),
$$
and $\tilde{Q}_{p-1}^\delta$ being the $H^1_{0,\Gamma_D}(\Omega)'$-bounded, efficiently applicable projector onto $\cS_{p-1}^{-1}(\tria^\delta)$ defined as the adjoint of the projector ``$P_\tria$'' from
 \cite[Thm.~5.1]{249.97}, or, alternatively for $p=1$, the projector ``$Q_h$'' from \cite[Prop.~8]{75.259}, let 
 \be \label{mildmodifiedrhs}
(\vec{p}^\delta,u^\delta):=\argmin_{(\vec{q},w) \in X^\delta} \tfrac12 \|G(\vec{q},w)-(0,\tilde{Q}_{p-1}^\delta g) \|^2_{L_2(\Omega)^d \times L_2(\Omega)}.
\ee

Let $P^\delta_{p-1} \in \cL\big(H_{0,\Gamma_N}(\divv;\Omega), H_{0,\Gamma_N}(\divv;\Omega)\big)$ be the projector onto $\RT_{p-1}(\tria^\delta) \cap H_{0,\Gamma_N}(\divv;\Omega)$ constructed in \cite{70.991}. 
It has a commuting diagram property (being the essence behind this approach), and consequently for $\vec{q} \in H_{0,\Gamma_N}(\divv;\Omega)$ with $\divv \vec{q} \in \cS_{p-1}^{-1}(\tria^\delta)$, it satisfies
$$
\|\vec{q} -P^\delta_{p-1} \vec{q}\|_{H(\divv;\Omega)} \lesssim \inf_{\vec{z} \in \RT^{-1}_{p-1}(\tria^\delta)}\|\vec{q} - \vec{z}\|_{L_2(\Omega)}.
$$

Let $(\vec{p},u)$ denote the solution of the mild-weak system $\vec{p}-A\nabla u=0$,\linebreak \mbox{$\int_\Omega \vec{p}\cdot\nabla v \,dx=g(v)$} ($v \in H^1_{0,\Gamma_D}(\Omega)$),
and let $(\undertilde{\vec{p}},\undertilde{u})$ denotes this solution with $g$ replaced by $\tilde{Q}_{p-1}^\delta g$.
Notice that $G(\undertilde{\vec{p}},\undertilde{u})=(0,\tilde{Q}_{p-1}^\delta g)$ and so $\divv \undertilde{\vec{p}} \in \cS_{p-1}^{-1}(\tria^\delta)$.
From $g \mapsto (\vec{p},u) \in \cL\big(H^1_{0,\Gamma_D}(\Omega)', L_2(\Omega)^d \times H^1_{0,\Gamma_D}(\Omega)\big)$, and the quasi-optimality of the MINRES discretization \eqref{mildmodifiedrhs}  in $H(\divv;\Omega) \times H^1(\Omega)$-norm, we infer that
\begin{align*}
\|\vec{p}&-\vec{p}^\delta\|_{L_2(\Omega)^d}  + \|u-u^\delta\|_{H^1(\Omega)}\\ 
& \lesssim  \|g-\tilde{Q}_{p-1}^\delta g\|_{H^1_{0,\Gamma_D}(\Omega)'}+\|\undertilde{\vec{p}}-\vec{p}^\delta\|_{H(\divv;\Omega)}+\|\undertilde{u}-u^\delta\|_{H^1(\Omega)}\\ 
&\lesssim  \|g-\tilde{Q}_{p-1}^\delta g\|_{H^1_{0,\Gamma_D}(\Omega)'}+ \inf_{(\vec{z},w) \in X^\delta} \|\undertilde{\vec{p}}-\vec{z}\|_{H(\divv;\Omega)}+\|\undertilde{u}-w\|_{H^1(\Omega)} \\ 
&\leq \|g-\tilde{Q}_{p-1}^\delta g\|_{H^1_{0,\Gamma_D}(\Omega)'}+\|\undertilde{\vec{p}}-P^\delta_{p-1} \undertilde{\vec{p}} \|_{H(\divv;\Omega)}+ \hspace*{-2em}\inf_{\rule{0mm}{3.5mm}w \in \cS^0_p(\tria^\delta) \cap H^1_{0,\Gamma_D}(\Omega)}\hspace*{-2em} \|\undertilde{u}-w\|_{H^1(\Omega)} \\ 
& \lesssim  \|g-\tilde{Q}_{p-1}^\delta g\|_{H^1_{0,\Gamma_D}(\Omega)'}+ \inf_{(\vec{z},w) \in X^\delta} \|\undertilde{\vec{p}}-\vec{z}\|_{L_2(\Omega)^d}+ \|\undertilde{u}-w\|_{H^1(\Omega)}\\ 
& \lesssim   \inf_{z \in \cS_{p-1}^{-1}(\tria^\delta)} \|g-z\|_{H^1_{0,\Gamma_D}(\Omega)'}+ \inf_{(\vec{z},w) \in X^\delta} \|\vec{p}-\vec{z}\|_{L_2(\Omega)^d}+ \|u-w\|_{H^1(\Omega)}\\  
& \lesssim   \inf_{\vec{z} \in \RT_{p-1}(\tria^\delta) \cap H_{0,\Gamma_N}(\divv;\Omega)} \|g+\divv \vec{z}\|_{H^1_{0,\Gamma_D}(\Omega)'}\\
&\hspace*{4cm}+ \inf_{(\vec{z},w) \in X^\delta} \|\vec{p}-\vec{z}\|_{L_2(\Omega)^d}+ \|u-w\|_{H^1(\Omega)}\\
& \lesssim \inf_{(\vec{z},w) \in X^\delta} \|\vec{p}-\vec{z}\|_{L_2(\Omega)^d}+ \|u-w\|_{H^1(\Omega)},
\end{align*}
where for the last inequality we have used that for $\vec{z} \in \RT_{p-1}(\tria^\delta) \cap H_{0,\Gamma_N}(\divv;\Omega)$ and $v \in H^1_{0,\Gamma_D}(\Omega)$,
$|g(v)+\int_\Omega \divv \vec{z}\,v\,dx|=|\int_\Omega (\vec{p}- \vec{z}) \cdot \nabla v\,dx|$.
We conclude quasi-optimality of $(\vec{p}^\delta,u^\delta) \in X^\delta$ w.r.t.~the $L_2(\Omega)^d \times H^1(\Omega)^d$-norm.
\end{remark}

\subsection{Inf-sup conditions for Example~\ref{ex:mild-weak}\ref{ex:mild-weak2} (mild-weak formulation)}
We take
$$
X^\delta:=\cS^{-1}_{p-1}(\tria^\delta)^d \times \cS^0_p(\tria^\delta).
$$
For simplicity we assume that $A=\identity$ and $B=0$, so that $G_2(\vec{q},w)=G_2(\vec{q})$.

Again the term $\|\gamma_D w -h_D\|_{H^{\frac12}(\Gamma_D)}^2=\|\gamma_D w -h_D\|_{\widetilde{H}^{-\frac12}(\Gamma_D)'}^2$
can be handled as in Example~\ref{ex:2ndorder}. The dual norm can be discretized by replacing $\widetilde{H}^{-\frac12}(\Gamma_D)$ by $\cS^{-1}_{p}(\cF^\delta_{\Gamma_D})$.

From $\int_\Omega \vec{q}\cdot \nabla v\,dx=\sum_{K \in \tria^\delta}\{\int_K -\divv \vec{q}\,v \,dx+\int_{\partial K}\vec{q}\cdot\vec{n}\, v\,ds\}$
where, when $p\geq 2$, for $K \in \tria^\delta$, $\divv \vec{q} \in \cP_{p-2}(K)$, and for $e \in \cF(\tria^\delta)$, $\vec{q}\cdot\vec{n} \in \cP_{p-1}(e)$,
we conclude that the term $\|G_2(\vec{q})-f_2\|_{H^1_{0,\Gamma_D}(\Omega)'}$ can be handled as in Example~\ref{ex:2ndorder}. The dual norm can be discretized by replacing $H^1_{0,\Gamma_D}(\Omega)$ by $\cS^0_{p+d-1}(\tria^\delta) \cap H^1_{0,\Gamma_D}(\Omega)$.

\begin{remark}[Approach from \cite{35.83}] Consider the mild-weak formulation with homogeneous essential boundary data $h_D=0$ (i.e., Example~\ref{ex:mild-weak}\ref{ex:mild-weak1}), 
as well as $h_N=0$, and, for simplicity,
$A=\identity$ and $B=0$.
Our approach was to determine $Y^\delta \subset H^1_{0,\Gamma_D}(\Omega)$ that allows for the construction of $\Pi^\delta \in \cL(H^1_{0,\Gamma_D}(\Omega),Y^\delta)$ with
$\int_\Omega \vec{q}\cdot \nabla(\identity-\Pi^\delta) v \,dx=0$ ($\vec{q} \in \cS_{p-1}^{-1}(\tria^\delta)^d,\,v \in H^1_{0,\Gamma_D}(\Omega)$).
Consequently, we could replace the term $\|v \mapsto \int_{\Omega}\vec{q}\cdot \nabla v \,dx-g(v)\|^2_{H^1_{0,\Gamma_D}(\Omega)'}$ in the least-squares minimization by 
the computable term $\|v \mapsto \int_{\Omega}\vec{q}\cdot \nabla v \,dx-g(v)\|^2_{{Y^\delta}'}$ without compromizing quasi-optimality of the resulting least-squares solution $(\vec{p}^\delta,u^\delta) \in X^\delta$.

Under the additional conditions that  $g \in L_2(\Omega)$, and that the finite element space $X^\delta$ w.r.t.~$\tria^\delta$ is contained in $H^1_{0,\Gamma_D}(\Omega) \times H_{0,\Gamma_N}(\divv;\Omega)$, for a finite element space $Y^\delta$ w.r.t.~$\tria^\delta$ for which there exists a mapping $\Pi^\delta \in \cL(H^1_{0,\Gamma_D}(\Omega),Y^\delta)$ with $\|h_\delta(\identity-\Pi^\delta)\|_{\cL(H^1_{0,\Gamma_D}(\Omega),L_2(\Omega))}\lesssim 1$, the approach from \cite{35.83} is to compute
$$
\argmin_{(\vec{q},w) \in X^\delta} \tfrac12\big(\|\vec{q}-\nabla w\|_{L_2(\Omega)^d}^2+\|v \mapsto \int_{\Omega}\vec{q}\cdot \nabla v - g v\,dx\|_{{Y^\delta}'}^2+
\|h_\delta (\divv \vec{q}+g)\|_{L_2(\Omega)}^2\big).
$$
So compared to our least-squares functional there is the additional term $\|h_\delta (\divv \vec{q}+g)\|_{L_2(\Omega)}^2$, whereas on the other hand the selection of $Y^\delta$ is less demanding.
Following \cite{35.83}, it can be shown that the resulting least squares solution denoted by $(\vec{p}^\delta,u^\delta)$ satisfies
\begin{align*}
\|\vec{p}-\vec{p}^\delta\|_{L_2(\Omega)^d}+&\|u-u^\delta\|_{H^1(\Omega)} \\
& \lesssim 
\inf_{(\vec{q},w) \in X^\delta} \|\vec{p}-\vec{q}\|_{L_2(\Omega)^d}+\|u-w\|_{H^1(\Omega)} +\|h_\delta \divv(\vec{p}-\vec{q})\|_{L_2(\Omega)}.
\end{align*}
This estimate does not imply quasi-optimality, but under usual regularity conditions w.r.t.~Hilbertian Sobolev spaces optimal rates can be demonstrated.
The assumption $g \in L_2(\Omega)$ can be weakened by replacing $g$ by an approximation from a finite element space w.r.t.~$\tria^\delta$.
\end{remark}

\subsection{Inf-sup condition for Example~\ref{ex:ultra-weak} (ultra-weak formulation)}
We restrict our analysis to the case that $|\Gamma_D|>0$, $A=\identity$, and $B=0$.
Then for $(\vec{q},w) \in X =L_2(\Omega)^d \times L_2(\Omega)$, and $(\vec{z},v) \in Y=H_{0,\Gamma_N}(\divv;\Omega) \times H^1_{0,\Gamma_D}(\Omega)$,
\be \label{eq:a}
(G(\vec{q},w))(\vec{z},v)=\int_\Omega \vec{q}\cdot\vec{z}+w \divv \vec{z}+\vec{q}\cdot \nabla v\,dx.
\ee
So far, for the lowest order case of
$$
X^\delta=\cS^{-1}_0(\tria^\delta)^d \times \cS^{-1}_0(\tria^\delta),
$$
we are able to construct a suitable Fortin interpolator taking
\be \label{eq:c}
Y^\delta=\big(\RT_{0}(\tria^\delta) \times \cS^0_d(\tria^\delta)\big) \cap Y.
\ee

We will utilise the Crouzeix-Raviart finite element space 
$$
\CR_{\Gamma_D}(\tria^\delta):=\{w \in \cS^{-1}_1(\tria^\delta)\colon \int_e[v]_e \,ds=0 \,\,(e \in \cF(\tria^\delta),e \not\subset \Gamma_N)\},
$$
where $[v]_e$ denotes the jump of $v$ over $e$ (with $v$ extended with zero outside $\Omega$).
With the abbreviation
$$
\RT_{\Gamma_N}(\divv0;\tria^\delta):=\RT_0(\tria^\delta) \cap H_{0,\Gamma_N}(\divv0;\tria^\delta),
$$
and with $\nabla_{\tria^\delta}$ denoting the piecewise gradient, we have the following generalisation of \cite[Thm.~4.1]{14.23} that was restricted to $d=2$ .

\begin{lemma}[discrete Helmholtz decomposition] \label{lem:helm}It holds that
$$
\cS_0^{-1}(\tria^\delta)^d=\RT_{\Gamma_N}(\divv0;\tria^\delta)  \oplus^{\perp_{L_2(\Omega)^d}} \nabla_{\tria^\delta} \CR_{\Gamma_D}(\tria^\delta). 
$$
\end{lemma}

\begin{proof} For $(\vec{q},w) \in \RT_{\Gamma_N}(\divv0;\tria^\delta) \times \CR_{\Gamma_D}(\tria^\delta)$, a piecewise integration-by-parts shows that
$$
\int_\Omega \vec{q} \cdot \nabla_{\tria^\delta} w\,dx=\sum_{e \in \cF(\tria^\delta)} \int_e [w]_e \,\vec{q}\cdot\vec{n}\,ds=0.
$$
It is known that, besides $\nabla_{\tria^\delta} \CR_{\Gamma_D}(\tria^\delta)$, also $\RT_{\Gamma_N}(\divv0;\tria^\delta)$ is in $ \cS_0^{-1}(\tria^\delta)^d$.

From $\divv\colon \RT_0(\tria^\delta) \cap H_{0,\Gamma_N}(\divv;\tria^\delta) \rightarrow \cS_0^{-1}(\tria^\delta)$, and $\dim \cS_0^{-1}(\tria^\delta)=\# \tria^\delta$, one infers
$$
\dim \RT_{\Gamma_N}(\divv0;\tria^\delta)\geq\# \cF(\tria^\delta)-\# \{e \in \cF(\tria^\delta)\colon e \subset \Gamma_N\}-\#\tria.
$$
From $\dim \CR_{\Gamma_D}(\tria^\delta)=\# \cF(\tria^\delta)-\#\{e \in \cF(\tria^\delta)\colon e \subset \Gamma_D\}$ and $\nabla_{\tria^\delta}$ being injective on $\CR_{\Gamma_D}(\tria^\delta)$, and
$(d+1) \#\tria^\delta=2 \#\cF(\tria^\delta)-\#\{e \in \cF(\tria^\delta)\colon e \subset \partial\Omega\}$,
we conclude that
$$
\dim \cS_0^{-1}(\tria^\delta)^d \leq \dim \nabla_{\tria^\delta} \CR_{\Gamma_D}(\tria^\delta) +\dim  \RT_{\Gamma_N}(\divv0;\tria^\delta),
$$
which completes the proof.
\end{proof}

\begin{theorem} \label{thm:important} For $G$, $X^\delta$, and $Y^\delta$ from \eqref{eq:a}-\eqref{eq:c}, it holds that
$$
\inf_{0 \neq (\vec{q},w) \in X^\delta}\frac{\sup_{0 \neq (\vec{z},v) \in Y^\delta} \frac{|(G(\vec{q},w))(\vec{z},v)|}{\|(\vec{z},v)\|_Y}}{\|G(\vec{q},w)\|_{Y'}} \gtrsim 1.\footnotemark[\value{footnoteValueSaver}]
$$
\end{theorem}

\begin{proof} We construct a Fortin interpolator $\Pi^\delta\colon Y \rightarrow Y^\delta$ of the form $\Pi^\delta(\vec{z},v)=(\Pi_1^\delta \vec{z},\Pi_2^\delta(\vec{z},v))$.

Let $P_0^\delta$ denote the $H(\divv;\Omega)$-bounded projector $H_{0,\Gamma_N}(\divv;\Omega) \rightarrow \RT_0(\tria^\delta) \cap H_{0,\Gamma_N}(\divv;\Omega)$ from \cite{70.991}, which has the commuting diagram property 
$$
\ran \divv (\identity-P_0^\delta) \perp_{L_2(\Omega)} \cS^{-1}_0(\tria^\delta).
$$
With $Q^\delta$ being the $L_2(\Omega)^d$-orthogonal projector onto $\RT_{\Gamma_N}(\divv0;\tria^\delta)$, we set $\Pi_1^\delta=P_0^\delta+Q^\delta(\identity-P_0^\delta) \in \cL\big(H_{0,\Gamma_N}(\divv;\Omega),\RT_0(\tria^\delta) \cap H_{0,\Gamma_N}(\divv;\Omega)\big)$.

Writing, for $(\vec{q},w) \in X^\delta$,  $\vec{q}=\vec{r}+\nabla_{\tria^\delta} t$, where $(\vec{r},t) \in \RT_{\Gamma_N}(\divv0;\tria^\delta) \times \CR_{\Gamma_D}(\tria^\delta)$, 
the definition of $\Pi_1^\delta$, Lemma~\ref{lem:helm}, and the fact that $H_{0,\Gamma_N}(\divv0;\Omega) \perp_{L_2(\Omega)^d} \nabla H_{0,\Gamma_D}^1(\Omega)$ show that
for $(\vec{z},v) \in Y$ it holds that
\begin{align} \nonumber
&(G(\vec{q},w))((\identity-\Pi)(\vec{z},v))\\  \nonumber
&=\int_\Omega (\vec{r}+\nabla_{\tria^\delta} t)\cdot\big((\identity -\Pi_1^\delta)\vec{z}+\nabla(v-\Pi_2^\delta(\vec{z},v)\big)+w\divv(\identity-\Pi_1^\delta)\vec{z}\,dx\\ \label{eq:rhs}
&=\int_\Omega \nabla_{\tria^\delta} t \cdot \big((\identity-P_0^\delta)\vec{z}+\nabla(v-\Pi_2^\delta(\vec{z},v))\big)\,dx.
\end{align}

It remains to define $\Pi_2^\delta(\vec{z},v) \in \cS_d^0(\tria^\delta) \cap H^1_{0,\Gamma_D}(\Omega)$ such that the last expression vanishes for all $t \in\CR_{\Gamma_D}(\tria^\delta)$ and $(\vec{z},v) \in Y$.
Let $\tilde{v} \in \CR_{\Gamma_D}(\tria^\delta)$ solve
$$
\int_\Omega \nabla_{\tria^\delta} t \cdot \nabla_{\tria^\delta} \tilde{v}\,dx=\int_\Omega \nabla_{\tria^\delta} t \cdot \big((\identity-P_0^\delta)\vec{z}+\nabla v)\big)\,dx \quad(t \in\CR_{\Gamma_D}(\tria^\delta)).
$$
It satisfies
$$
\|\nabla_{\tria^\delta} \tilde{v}\|_{L_2(\Omega)^d} \leq \|(\identity -P_0^\delta)\vec{z}\|_{L_2(\Omega)}+|v|_{H^1(\Omega)} \lesssim \|\vec{z}\|_{H(\divv;\Omega)}+|v|_{H^1(\Omega)}.
$$
There exists a conforming companion operator $E_{\tria^\delta}\colon \CR_{\Gamma_D}(\tria^\delta) \rightarrow \cS^0_d(\tria^\delta) \cap H^1_{0,\Gamma_D}(\Omega)$ with $\ran (\nabla E_{\tria^\delta}-\nabla_{\tria^\delta})\perp_{L_2(\Omega)^d} \cS^{-1}_0(\tria^\delta)$, and $\|\nabla E_{\tria^\delta} \cdot\|_{L_2(\Omega)^d} \lesssim \|\nabla_{\tria^\delta} \cdot\|_{L_2(\Omega)}$ on $\CR_{\Gamma_D}(\tria^\delta)$
(one can take the operator $J_2$ from \cite[Proof of Prop.~2.3]{37.46}, see \cite{37.475} for a generalisation to $d \geq 2$).
Defining $\Pi_2^\delta(\vec{z},v):=E_{\tria^\delta} \tilde{v}$, we conclude that \eqref{eq:rhs} vanishes for all $t \in\CR_{\Gamma_D}(\tria^\delta)$, and 
that $\|\Pi_2^\delta(\vec{z},v)\|_{H^1(\Omega)} \lesssim \|\vec{z}\|_{H(\divv;\Omega)}+\|v\|_{H^1(\Omega)}$, so that $\Pi^\delta \in \cL(Y,Y^\delta)$ is a valid Fortin interpolator.
\end{proof}

\begin{remark} Although $G=(G_1,G_2) \in \Lis\big(X,H_{0,\Gamma_N}(\divv;\Omega) \times H^1_{0,\Gamma_D}(\Omega)\big)$, in this subsection we did not verify inf-sup stability for $G_1$ and $G_2$ separately to conclude inf-sup stability for $G$ by Lemma~\ref{lem1}.  The reason is that we did not manage to verify inf-sup stability for $G_1(\vec{q},w)(\vec{z})=\int_\Omega \vec{q}\cdot\vec{z} +w \divv \vec{z}\, dx$.
We notice that in the context of a DPG method, in \cite[Sect.~3]{75.61} inf-sup stability has been demonstrated separately for $G_1$ and $G_2$, even for trial spaces $X^\delta$ of general polynomial degree.
\end{remark}

\subsection{Preconditioners}
At several places, it was desirable or, in case of fractional norms,  even essential to have an efficiently evaluable (uniform) preconditioner $K^\delta={K^\delta}' \in \Lis({Z^\delta}',Z^\delta)$ available, where $Z^\delta$ was of one of the following types:
{\renewcommand{\labelenumi}{(\roman{enumi})}%
\begin{enumerate}
\item \label{case1} $\cS^0_p(\tria^\delta)$ or $\cS^0_p(\tria^\delta) \cap H^1_{0,\Gamma_D}(\Omega)$ equipped with $\|\cdot\|_{H^1(\Omega)}$,
\item \label{case2} $\cS^{-1}_p(\{e \in \cF(\tria^\delta)\colon e \subset \Gamma_D\})$ equipped with $\|\cdot\|_{\widetilde{H}^{-\frac12}(\Gamma_D)}$,
\item \label{case3} $\cS^{0}_p(\{e \in \cF(\tria^\delta)\colon e \subset \Gamma_N\})$ equipped with $\|\cdot\|_{H_{00}^{\frac12}(\Gamma_N)}$,
\item \label{case4} $\RT_0(\tria^\delta) \cap H_{0,\Gamma_N}(\divv;\Omega)$ equipped with $\|\cdot\|_{H(\divv;\Omega)}$.
\end{enumerate}%
}

When $\tria^\delta$ is constructed from recurrent refinements by a fixed refinement rule starting from a fixed coarse partition, multi-level preconditioners of linear computational complexity are available for all four cases (see \cite{75.258} for 
Case~(ii), 
and \cite{14.24,14.26} or \cite{138.295} for
Case~(iv)). 
Alternatives for the fractional Sobolev norms are provided by `operator preconditioners' (see \cite{138.26,249.97,249.975}).

\section{Numerical experiments} \label{sec:numerics}
On a square domain $\Omega = (0,1)^2$ with  Neumann and Dirichlet boundaries $\Gamma_N=\{0\}\times [0,1]$ and $\Gamma_D=\overline{\partial \Omega \setminus \Gamma_N}$, for $g \in H^1_{0,\Gamma_D}(\Omega)'$, $h_D \in H^{\frac12}(\Gamma_D)$, and $h_N \in H^{-\frac12}(\Gamma_N)$
we consider the Poisson problem of finding $u \in H^1(\Omega)$ that satisfies
$$
 \left\{
\begin{array}{r@{}c@{}ll}
-\Delta u&\,\,=\,\,& g &\text{ on } \Omega,\\
u &\,\,=\,\,& h_D &\text{ on } \Gamma_D,\\
\nabla u \cdot \vec{n}&\,\,=\,\,& h_N &\text{ on } \Gamma_N.
\end{array}
\right.
$$
In particular, we take $g=0$, $h_D(x,y) = \cos \tfrac{\pi x}{2}$, and $h_N=1$. Hence because of the incompatibility of the Dirichlet and Neumann data at $\Gamma_D \cap \Gamma_N$, the pair of the gradient of the solution and the solution
 $(\vec{p},u):=(\nabla u,u)$
has (mild) singularities at the points $(0,0)$ and $(0,1)$, see Figure~\ref{fig:solutions}.
\begin{figure}[h!]
\hspace*{0cm}
\begin{subfigure}{0.5\textwidth}
\centering
\includegraphics[width=\linewidth]{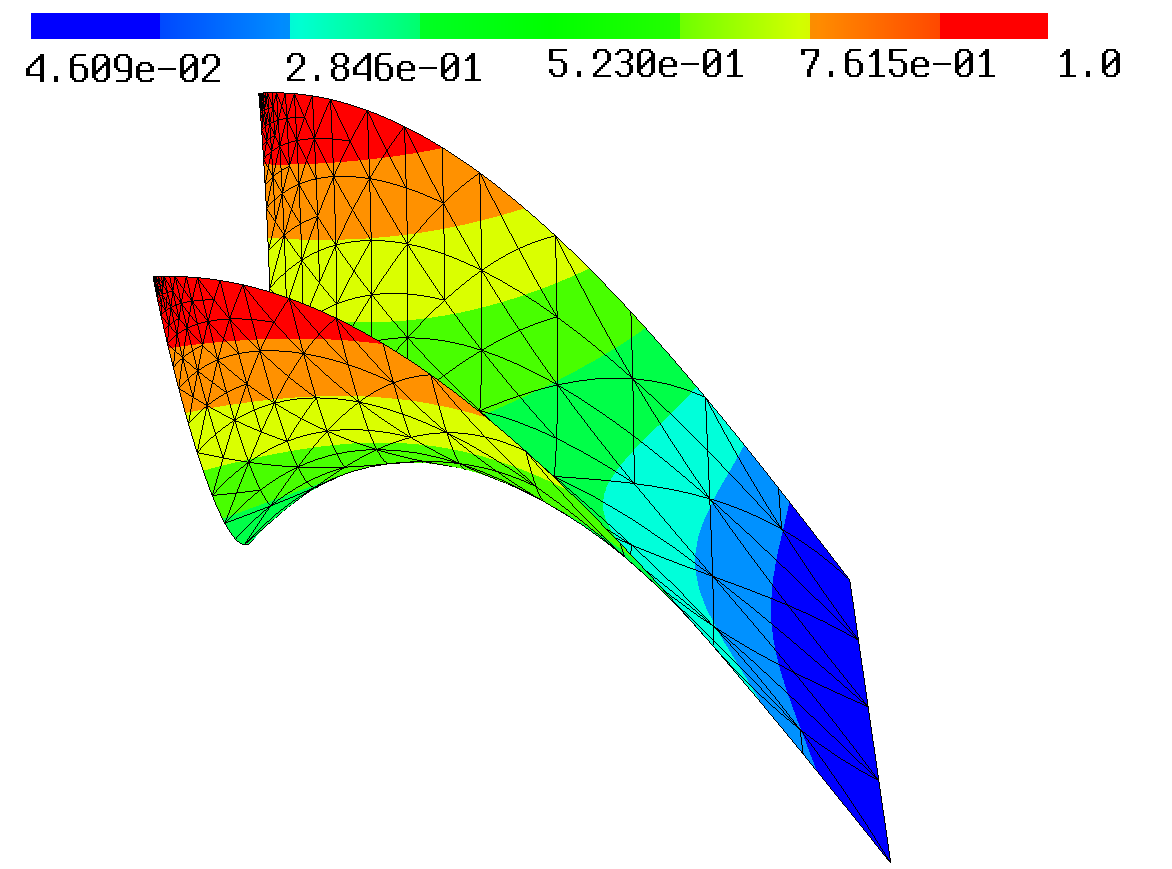}

\end{subfigure}\hspace*{0cm}%
\begin{subfigure}{0.5\textwidth}
\centering
\includegraphics[width=\linewidth]{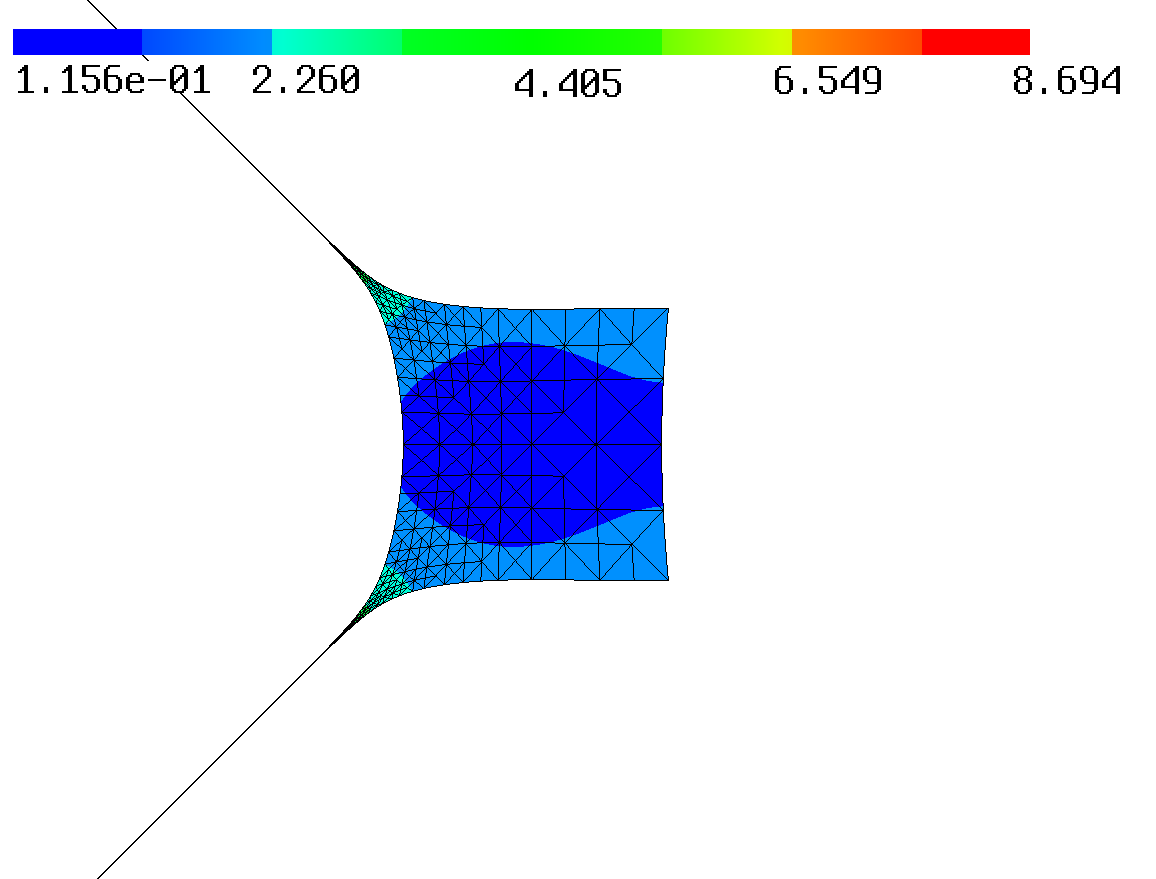}
\end{subfigure}
\caption{Left: a plot of $u^\delta$; right: a plot of $|\vec{p}^\delta|$.}
\label{fig:solutions}
\end{figure}

We consider above problem in the first order ultra-weak formulation from Example \ref{ex:ultra-weak}.
We consider a family of conforming triangulations $\{\mathcal{T}^\delta\}_{\delta}$ of $\Omega$, where each triangulation is created using newest vertex bisections starting from an initial triangulation that consists of 4 triangles created by cutting $\Omega$ along its diagonals. The interior vertex of the initial triangulation is labelled as the `newest vertex' of all four triangles in this initial mesh.
Given some polynomial degree $p \in N_0$, we set
$$
X^\delta:=\cS^{-1}_p(\tria^\delta)^d \times \cS^{-1}_p(\tria^\delta).
$$

With $(G(\undertilde{\vec{p}},\undertilde{u}))(\vec{\mu},\lambda):=\int_\Omega \undertilde{\vec{p}}\cdot \vec{\mu}+\undertilde{u} \divv \vec{\mu}+\undertilde{\vec{p}}\cdot \nabla \lambda\,dx$, for a suitable finite dimensional subspace $Y^\delta=Y^\delta(X^\delta) \subset Y:=H_{0,\Gamma_N}(\divv;\Omega) \times H^1_{0,\Gamma_D}(\Omega)$ the practical MINRES  method computes
$(\vec{p}^\delta,u^\delta,\vec{\mu}^\delta,\lambda^\delta) \in X^\delta \times Y^\delta$ such that 
\begin{align*}
	\langle (\vec{\mu}^\delta, \lambda^\delta),(\undertilde{\vec{\mu}},\undertilde{\lambda})\rangle&_{H(\divv;\Omega)\times H^1(\Omega)}  +  (G(\vec{p}^\delta,u^\delta))(\undertilde{\vec{\mu}},\undertilde{\lambda}) +(G(\undertilde{\vec{p}}, \undertilde{u}))(\vec{\mu}^\delta, \lambda^\delta)  \\&\qquad=\int_{\Gamma_D} h_D \undertilde{\vec{\mu}}\cdot \vec{n}\,ds+ g( \undertilde{\lambda})  + \int_{\Gamma_N} h_N \undertilde{\lambda} \,ds=: f(\undertilde{\vec{\mu}},\undertilde{\lambda})
\end{align*}
for all $(\undertilde{\vec{p}},\undertilde{u},\undertilde{\vec{\mu}},\undertilde{\lambda}) \in X^\delta \times Y^\delta$.

As we have seen, when $Y^\delta$ is selected such that 
$$
\gamma^\delta = \inf_{0 \neq(\undertilde{\vec{p}},\undertilde{u}) \in X^\delta}\frac{\sup_{0 \neq (\vec{\mu},\lambda) \in Y^\delta} \frac{|(G (\undertilde{\vec{p}},\undertilde{u})(\vec{\mu},\lambda)|}{\|(\vec{\mu},\lambda)\|_{Y}}}{\|G (\undertilde{\vec{p}},\undertilde{u})\|_{Y'}}\, \gtrsim 1,
$$
then $(\vec{p}^\delta,u^\delta)$ is a quasi-best approximation from $X^\delta$ to $(\vec{p},u)$ w.r.t.~the norm on $X:=L_2(\Omega)^d \times L_2(\Omega)$.

For $p \in \N_0$, we take
$$
Y^\delta:=\big(\RT_{p}(\tria^\delta) \times \cS^0_{d+p}(\tria^\delta)\big) \cap Y, 
$$
where thus $d=2$.
Theorem~\ref{thm:important} shows that for $p=0$ above uniform inf-sup condition is satisfied.
Using that, thanks to $G \in \Lis(X,Y')$, 
$$
\gamma^\delta \eqsim \tilde{\gamma}^\delta:=\inf_{0 \neq(\undertilde{\vec{p}},\undertilde{u}) \in X^\delta}
\sup_{0 \neq (\vec{\mu},\lambda) \in Y^\delta} 
\frac{|(G (\undertilde{\vec{p}},\undertilde{u})(\vec{\mu},\lambda)|}{\|(\vec{\mu},\lambda)\|_Y \|(\undertilde{\vec{p}},\undertilde{u})\|_X},
$$
for $p \in \{1,2,3,4\}$ we verified numerically whether our choice of $Y^\delta$ gives inf-sup stability. 
The results given in Figure~\ref{fig:ultra-weakinfsup} indicate that this is the case. 
\begin{figure}[h!]
\hspace*{0cm}
\begin{subfigure}{0.5\textwidth}
\centering
\includegraphics[width=\linewidth]{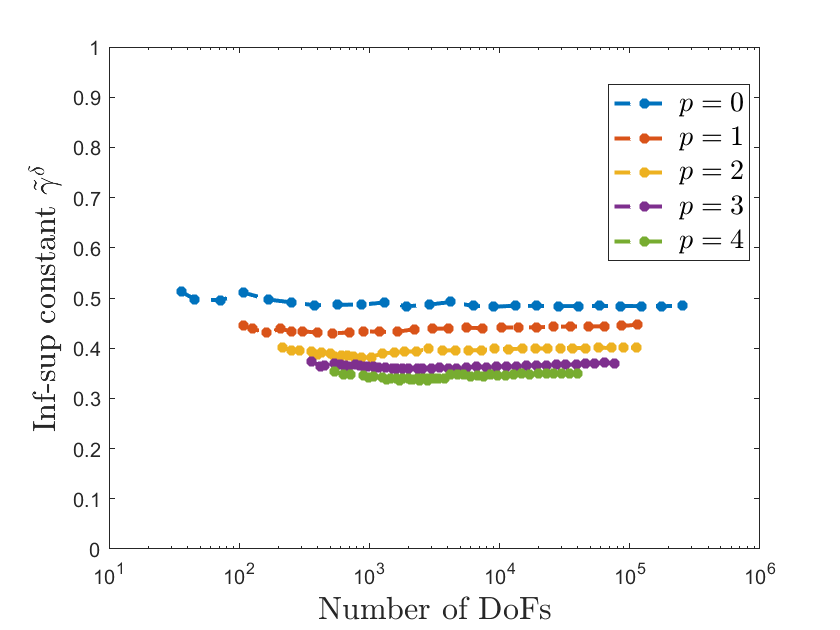}

\end{subfigure}\hspace*{0cm}%
\begin{subfigure}{0.5\textwidth}
\centering
\includegraphics[width=\linewidth]{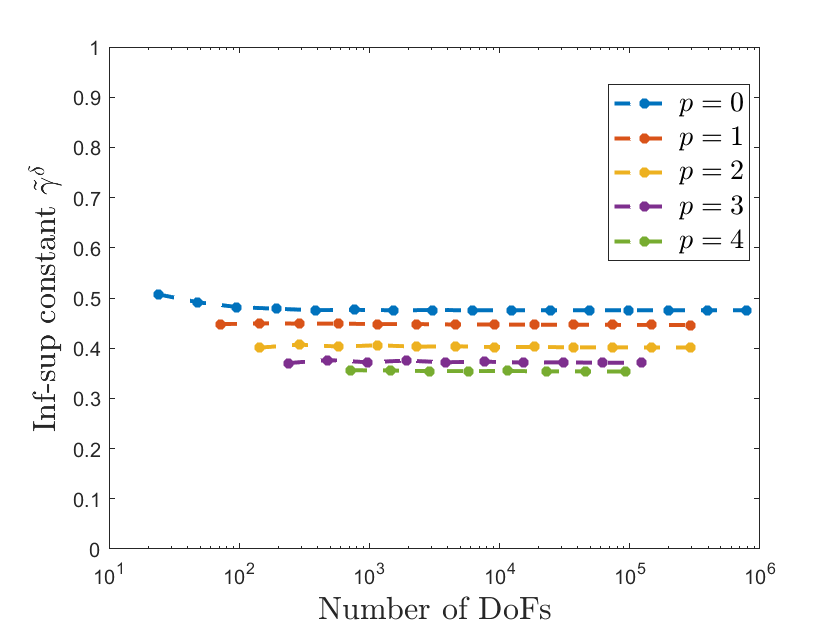}
\end{subfigure}

\caption{Number of DoFs in $X^\delta$ vs.~$\tilde\gamma^\delta$. Left: triangulations that are locally refined towards $(0,0)$ and $(0,1)$; right: uniform triangulations.}
\label{fig:ultra-weakinfsup}
\end{figure}

The practical MINRES  method comes with a built-in a posteriori error estimator given by $\cE(\vec{p}^\delta,u^\delta,f)=\sqrt{\sum_{T \in \tria^\delta} \|\vec{\mu}^\delta\|_{H(\divv;T)}^2+\|\lambda^\delta\|_{H^1(T)}^2}$ (see Remark~\ref{rem:comp}). For $p \in \{0,1,2,3\}$ we performed numerical experiments with uniform and adaptively refined triangulations. Concerning the latter, we have used the element-wise error indicators $\sqrt{\|\vec{\mu}^\delta\|_{H(\divv;T)}^2+\|\lambda^\delta\|_{H^1(T)}^2}$ to drive an AFEM with D\"{o}rfler marking with marking parameter $\theta=0.6$.
We have seen that the estimator $\cE(\vec{p}^\delta,u^\delta,f)$ is efficient, but because the data-oscillation term can be of the order of the best approximation error, it is not necessarily reliable.
Therefore instead of using the a posteriori error estimator to assess the quality of our MINRES method, as a measure for the error 
we computed the $X$-norm of the difference with the MINRES solution for $p=4$ on the same triangulation, denoted as $(\vec{p}_4^\delta,u_4^\delta)$. The results given in Figure~\ref{fig2} show that for uniform refinements increasing $p$ does not improve the order of convergence, due to the limited regularity of
 the solution in the Hilbertian Sobolev scale. The results indicate that the solution is just in $H^2(\Omega)$.
\begin{figure}[h!]
\hspace*{0cm}
\begin{subfigure}{0.5\textwidth}
\centering
\includegraphics[width=\linewidth]{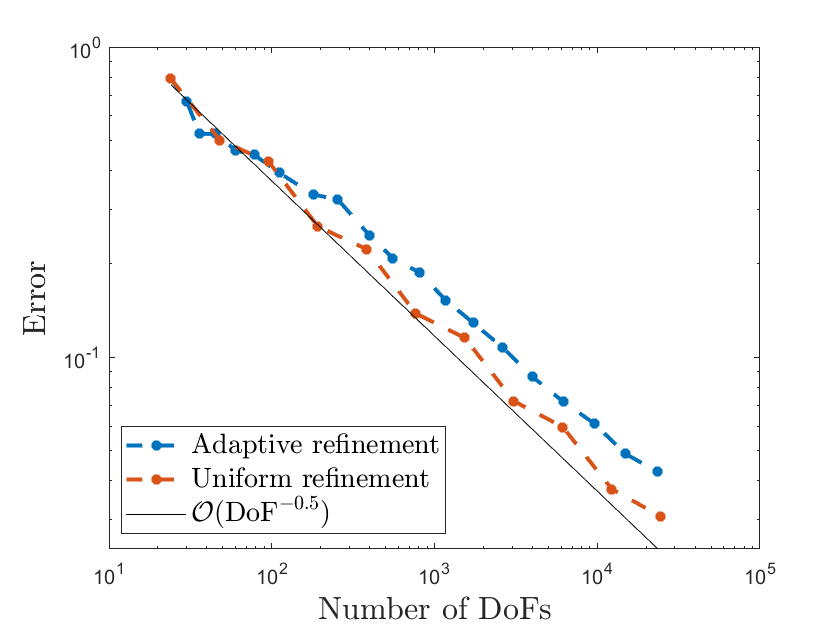}

\end{subfigure}\hspace*{0cm}%
\begin{subfigure}{0.5\textwidth}
\centering
\includegraphics[width=\linewidth]{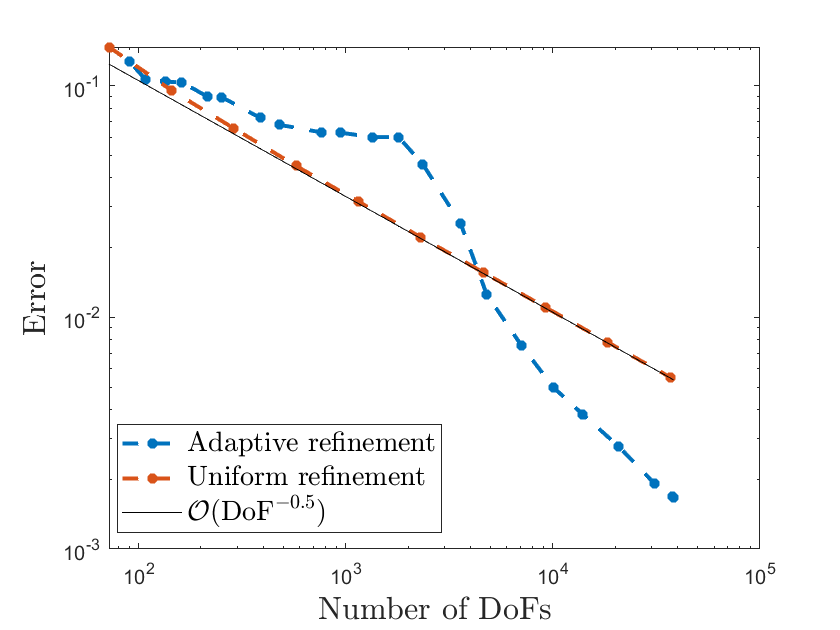}
\end{subfigure}
\begin{subfigure}{0.5\textwidth}
\centering
\includegraphics[width=\linewidth]{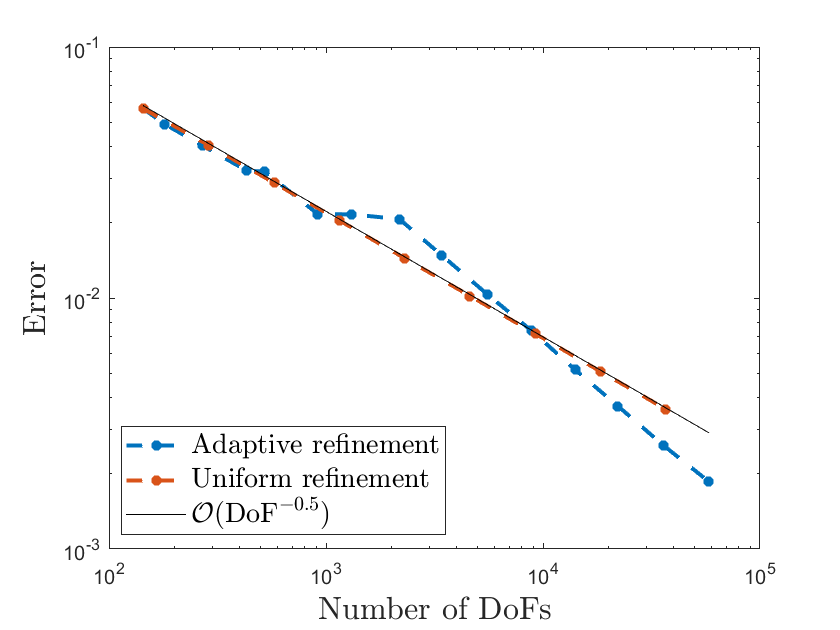}

\end{subfigure}\hspace*{0cm}%
\begin{subfigure}{0.5\textwidth}
\centering
\includegraphics[width=\linewidth]{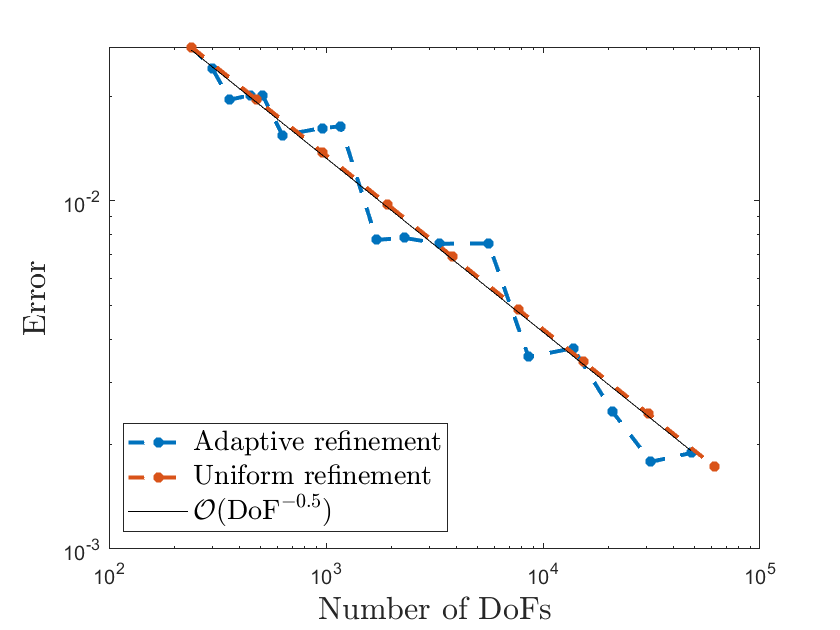}
\end{subfigure}
\caption{Number of DoFs in $X^\delta$ vs.~$\|(\vec{p}_4^\delta,u_4^\delta)-(\vec{p}^\delta,u^\delta)\|_X$. Left-upper: $p=0$, right-upper: $p=1$, left-bottom: $p=2$, right-bottom: $p=3$.}
\label{fig2}
\end{figure}
Furthermore we see that adaptivity does not yield improved convergence rates. We expect that the reason  for the latter is that, with our current choice of $Y^\delta$, the data oscillation term dominates our error estimator, so that the local error indicators do not provide the correct information where to refine.

For this reason, we repeat the experiment from Figure~\ref{fig2} using the higher order test space
$$
Y^\delta:=\big(\RT_{p+1}(\tria^\delta) \times \cS^0_{d+p+1}(\tria^\delta)\big) \cap Y.
$$
Now we observe that the a posteriori error estimator is proportional (and actually quite close) to the error notion $\|(\vec{p}_4^\delta,u_4^\delta)-(\vec{p}^\delta,u^\delta)\|_X$,
and so we expect it indeed to be also reliable. In Figure~\ref{fig3} we give the number of DoFs vs.~$\cE(\vec{p}^\delta,u^\delta,f)$. 
\begin{figure}[h!]
\hspace*{0cm}
\begin{subfigure}{0.5\textwidth}
\centering
\includegraphics[width=\linewidth]{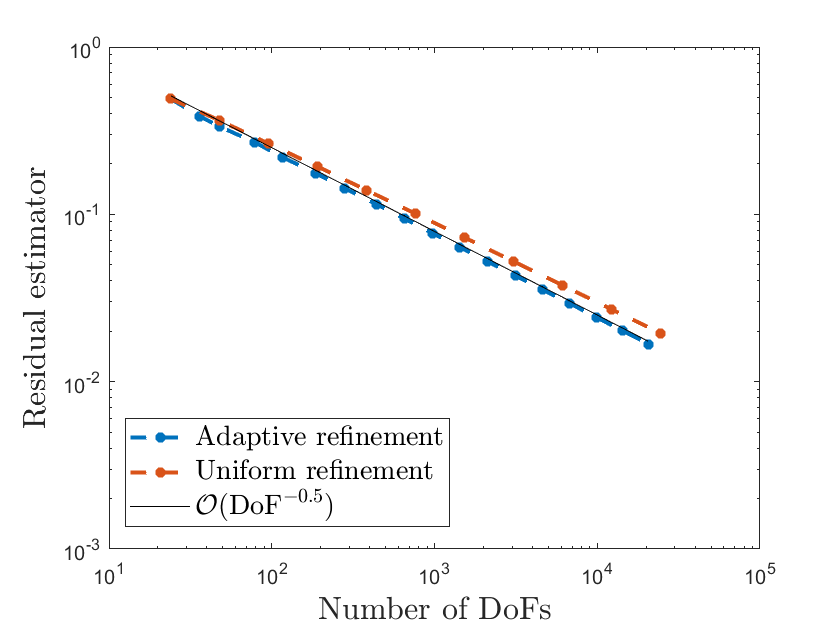}

\end{subfigure}\hspace*{0cm}%
\begin{subfigure}{0.5\textwidth}
\centering
\includegraphics[width=\linewidth]{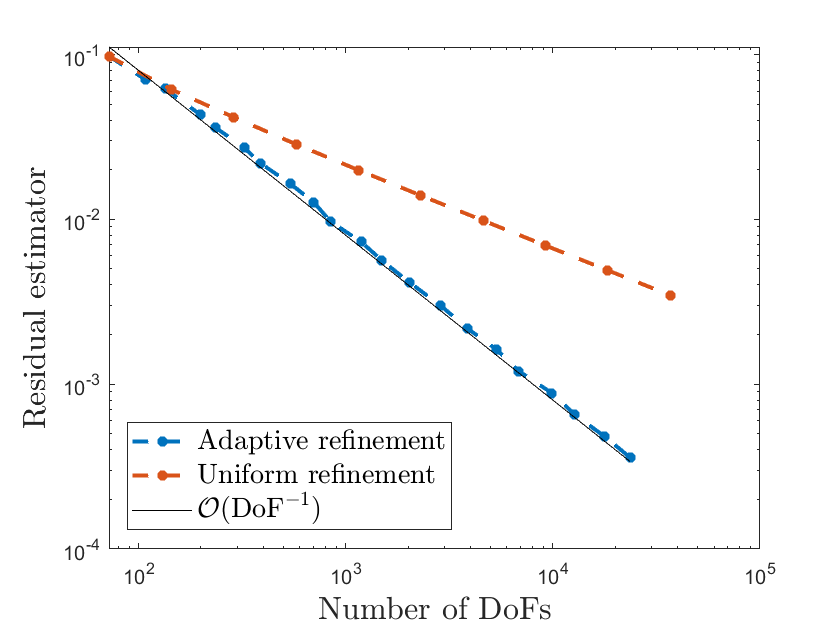}
\end{subfigure}

\begin{subfigure}{0.5\textwidth}
\centering
\includegraphics[width=\linewidth]{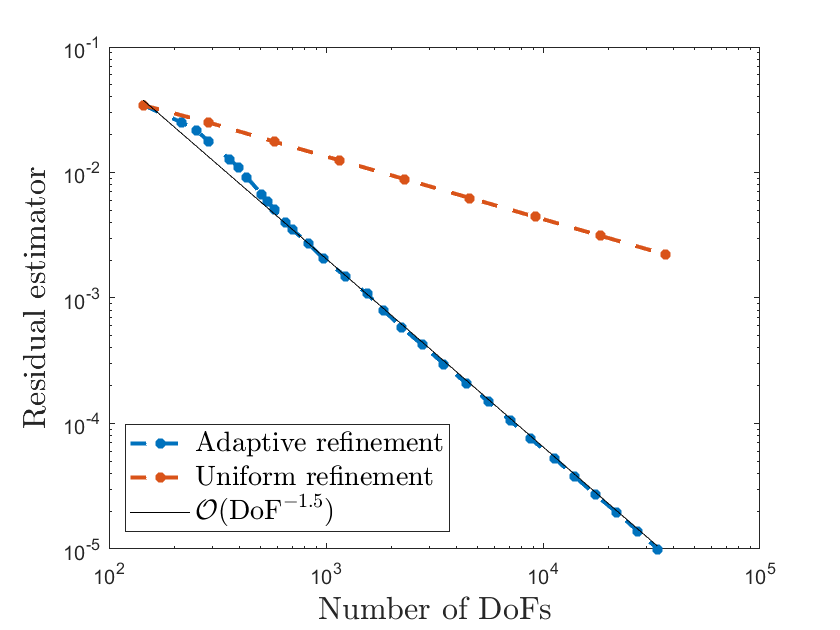}

\end{subfigure}\hspace*{0cm}%
\begin{subfigure}{0.5\textwidth}
\centering
\includegraphics[width=\linewidth]{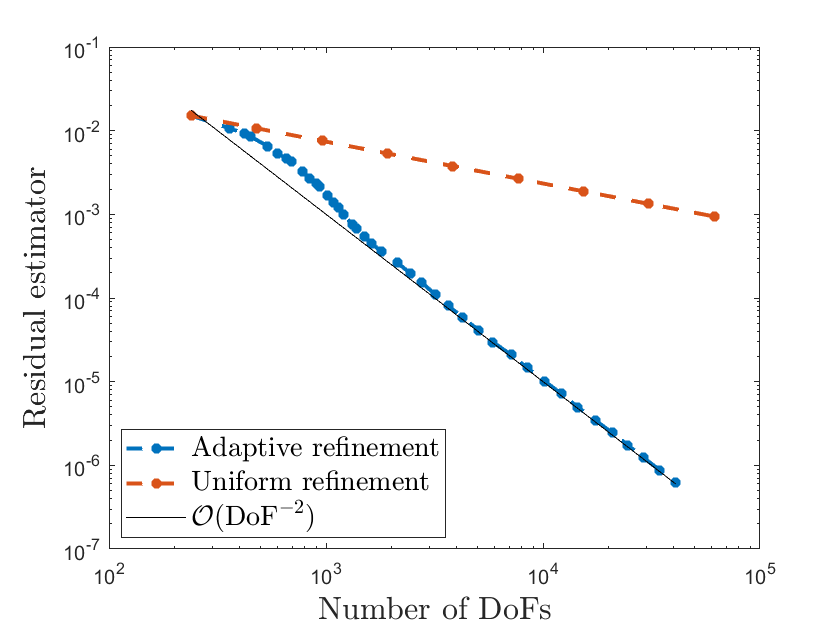}
\end{subfigure}
\caption{Number of DoFs in $X^\delta$ vs.~$\cE(\vec{p}^\delta,u^\delta,f)$. 
 Left-upper: $p=0$, right-upper: $p=1$, left-bottom: $p=2$, right-bottom: $p=3$.}
\label{fig3}
\end{figure}
As expected, the rates for uniform refinements are as before, but now we observe for the adaptive routine the generally best possible rates allowed by the order of approximation of $X^\delta$.

\section{Conclusion} \label{sec:conclusion}
In MINRES discretisations of PDEs often parts of the residual are measured in fractional or negative Sobolev norms.
In this paper a general approach has been presented to turn such an `impractical' MINRES method into a practical one, without compromizing quasi-optimality of the obtained numerical approximation, assuming that the test space that is employed is chosen such that a (uniform) inf-sup condition is valid.
The resulting linear system is of a symmetric saddle-point form, but can be replaced by a symmetric positive definite system by the application of a (uniform) preconditioner at the test space, while still preserving quasi-optimality.
For four different formulations of scalar second order elliptic PDEs, the aforementioned uniform inf-sup condition has been verified for pairs of finite element trial and test spaces.
Numerical results have been presented for an ultra-weak first order system formulation of Poisson's problem that allows for a very convenient treatment of inhomogeneous mixed Dirichlet and Neumann boundary conditions.

\end{document}